\documentclass[a4paper, reqno]{amsart}
\usepackage{graphicx} 

\usepackage[margin=3cm]{geometry}
\usepackage{latexsym}
\usepackage{tcolorbox}
\usepackage{soul}
\usepackage{etoolbox}

\usepackage[utf8]{inputenc}
\usepackage[T1]{fontenc}

\usepackage{amssymb, amsmath, amsthm, amsrefs, xcolor}
\usepackage{mathtools}
\usepackage{hyperref}
\usepackage{mathrsfs}
\usepackage[matha, mathx]{mathabx}
\usepackage[shortlabels]{enumitem}

\newtheorem{thm}{Theorem}[section]
\newtheorem{corollary}[thm]{Corollary}
\newtheorem{lemma}[thm]{Lemma}
\newtheorem{proposition}[thm]{Proposition}

\theoremstyle{definition}

\newtheorem{remark}[thm]{Remark}

\newtheorem*{Que}{Question}

\newcommand{\R}{{\mathbb R}}
\newcommand{\N}{{\mathbb N}}
\newcommand{\C}{{\mathbb C}}

\newcommand{\dd}{{\mathrm d}}

\newcommand{\mG}{\mathcal{G}}
\newcommand{\tmG}{\tilde{\mathcal{G}}}
\newcommand{\mB}{\mathcal{B}}
\newcommand{\mS}{\mathcal{S}}
\newcommand{\mA}{\mathcal{A}}

\newcommand{\Fm}{\mathcal{A}}
\newcommand{\mF}{\mathcal{F}}
\newcommand{\lm}{\lambda}

\begin{document}

\numberwithin{equation}{section}
\title[High dimensional maximal functions for Gaussians, balls, and spheres]{On high dimensional maximal functions associated to Gaussians, balls, and spheres}

\author{Valentina Ciccone}
\address[Valentina Ciccone]{Institute of Mathematics, Polish Academy of Sciences, Śniadeckich 8, 00-656 Warszawa, Poland}
\email{vciccone@impan.pl}

\author{B{\l}a{\.z}ej Wr{\'o}bel}
\address[B{\l}a{\.z}ej Wr{\'o}bel]{
Institute of Mathematics
	of the Polish Academy of Sciences\\
	Śniadeckich 8\\
	00-656 Warsaw\\
	Poland \& Institute of Mathematics\\
	University of Wroc{\l}aw\\
	Plac Grun\-waldzki 2\\
	50-384 Wroc{\l}aw\\
	Poland}
\email{blazej.wrobel@math.uni.wroc.pl}

\keywords{Gaussian maximal operator, Hardy-Littlewood maximal operator, spherical maximal operator, high dimensional phenomena\\
\indent This article was merged into \url{https://arxiv.org/abs/2607.06041}}

\begin{abstract}
We prove that for each $p\in (1,\infty),$ the norms on $L^p(\mathbb{R}^d)$  of the maximal functions associated to Gaussians (heat semigroup), balls (Hardy-Littlewood averages), and spheres (spherical averages) converge, as the dimension $d\to \infty,$ to the same quantity $\lambda(p)$.
This is derived from the fact that the norms on $L^2(\mathbb{R}^d)$ of the maximal functions corresponding to the differences of Gaussian, ball, and spherical averages converge to zero with the dimension $d.$ The fact is proved with the aid of estimates for Fourier multiplier symbols corresponding to these averages, 
 a general principle that allows us to control the norm of a maximal function corresponding to a Fourier multiplier operator by the norm of the multiplier operator itself, and concentration properties of high dimensional Gaussian random vectors. 
Moreover, relying on the properties of the $d$-dimensional maximal function for the heat semigroup $\mathcal{G}_\ast^d$, we show that $\lambda(p)$ satisfies
$$
\frac25\frac{p}{p-1}\le\|\mathcal{G}_\ast^1\|_{L^p(\mathbb{R})\rightarrow L^p(\mathbb{R})}\le \lambda(p)\le \frac{p}{p-1}.
$$
In particular, to obtain the middle inequality we show that the norms on $L^p(\mathbb{R}^d)$ of the maximal function for the heat semigroup are non-decreasing in $d.$
\end{abstract}

\maketitle

\section{Introduction}

\subsection{Statement of the main results} 
For $t>0$ and for every $x\in\mathbb{R}^d$ we define the Hardy-Littlewood  averaging operator
\begin{align*}
    \mB_tf(x)=\mB_t^df(x)= \frac{1}{|B_t|}\int_{B_t}f(x-y)\dd y~, \qquad f\in L^1_{loc}(\mathbb{R}^d)~,
\end{align*}
where $B_t=\lbrace x\in\mathbb{R}^d: \, t^{-1}x\in B\rbrace$, $B$ is the Euclidean ball in $\mathbb{R}^d$ with center at the origin and radius one, and $|U|$ denotes the $d$-dimensional Lebesgue volume of a set $\subset\mathbb{R}^d$.
Similarly, we define the spherical averaging operator
\begin{align*}
\mS_tf(x)=\mS^d_tf(x)= \frac{1}{\sigma(\mathbb{S}^{d-1})}\int_{\mathbb{S}^{d-1}}f(x-t\omega)\dd \sigma(\omega)~, \qquad f\in\mathcal{S}(\mathbb{R}^d)~.
\end{align*}
Here $\sigma$ is the surface measure on the $(d-1)$-dimensional unit sphere. We let $\mG_t$ be the classical heat semigroup on $\R^d$ given by convolution with the Gaussian kernel  
$$\mG_tf(x)= \mG_t^d f(x)=\frac{1}{(4\pi t)^{d/2}}\int_{\R^d}e^{-|y|^2/(4t)} f(x-y)\dd y~.$$ We shall also need a re-parametrization of $\mG_t$ given by
\[
\tmG_t=\mG_{t/(2d)}.
\]

Let $\mA_t,$ $t>0,$ be any of the above averaging operators or their differences and denote by $\mA_*$ the corresponding maximal operator
\[
\mA_* f(x)=\sup_{t>0}|\mA_t f(x)|.
\]
Note that then $\mG_*=\tmG_*.$ For each  $p\in (1,\infty]$ we let
$\|\mA_*\|_{L^p(\R^d)}$ be the operator norm on $L^p(\R^d)$ of $\mA_*$. In other words $\|\mA_*\|_{L^p(\R^d)}$ is the best constant in the inequality
\[
\Vert \sup_{t>0} |\mA_t f|\Vert_{L^p(\R^d)}\le \|\mA_*\|_{L^p(\R^d)}\|f\|_{L^p(\R^d)}.
\]
It is well known that $\|\mB_*\|_{L^p(\R^d)}$ and $\|\mG_*\|_{L^p(\R^d)}$ are finite for all $d\in \N$ and $p\in (1,\infty).$ Furthermore, $\|\mS_*\|_{L^p(\R^d)}$ is finite for $p> d/(d-1).$ This is due to Stein \cite{St_spher} in dimensions $d\ge 3$ and due to Bourgain \cite{Bo1} in dimension $d=2.$ We remark that for each fixed $p\in (1,\infty)$ we have $p>d/(d-1)$ when $d$ is large enough and in particular $\|\mS_*\|_{L^p(\R^d)}$ is finite when $d\to \infty.$

In view of the pointwise estimate 
\begin{equation}
\label{eq: point est}
    \sup_{t>0} |\mG_t^d f(x)| \; \leq \sup_{t>0} \; |\mB_t^df(x)| \leq \; \sup_{t>0} \; |\mS^d_tf(x)| ~, 
\end{equation}
it immediately follows that
\begin{align}\label{eq_gauss_leq_ball_leq_sphere}
   \|\mG_*\|_{L^p(\R^d)} \, \leq \,  \|\mB_*\|_{L^p(\R^d)} \, \leq \,    \|\mS_*\|_{L^p(\R^d)} ~.
\end{align}
Our main goal in this work is to show that, in the limit as $d$ tends to infinity, \eqref{eq_gauss_leq_ball_leq_sphere} becomes an equality. This will be a consequence of the stronger statement in \eqref{eq: diff Lp}, which concerns the behavior of the maximal functions corresponding to the differences of Gaussian, ball, and spherical averages.
\begin{thm}
    \label{thm: limit}
    For each $p\in (1,\infty)$ we have
    \begin{equation}
 \label{eq: diff Lp}
 \lim_{d\to \infty}\Vert  (\mS-\tmG)_*\Vert_{L^p(\mathbb{R}^d)}= \lim_{d\to \infty}\Vert  (\mS-\mB)_*\Vert_{L^p(\mathbb{R}^d)}=\lim_{d\to \infty}\Vert  (\tmG-\mB)_*\Vert_{L^p(\mathbb{R}^d)}=0
 \end{equation}
and, consequently, 
    \begin{equation}
    \label{eq: limno1}
    \lim_{d\to\infty}   \|\mG_*\|_{L^p(\R^d)} =\lim_{d\to\infty}   \|\mB_*\|_{L^p(\R^d)}=\lim_{d\to\infty}   \|\mS_*\|_{L^p(\R^d)},
    \end{equation}
    i.e.\ the above limits exist and are all equal to the same quantity $\lm(p).$ Furthermore,   $\lm(p)$ satisfies  
    \begin{equation}
    \label{eq: lm bounds}
   \max\left\lbrace \frac25 \frac{p}{p-1}, 1 \right\rbrace \le \|\mG_\ast^1\|_{L^p(\mathbb{R})}\le \lm(p)\le \frac{p}{p-1}~.
    \end{equation}
\end{thm}

From the work of Stein \cite{St82} and Stein and Strömberg  \cite{SS83} one may deduce
 dimension-free estimates for each of the maximal functions under consideration. Namely, for each $p\in (1,\infty)$ we have
\begin{equation}
\label{eq: dimfree}
  \limsup_{d\to\infty}   \|\mG_*\|_{L^p(\R^d)} <\infty,\quad\limsup_{d\to\infty}   \|\mB_*\|_{L^p(\R^d)}<\infty,\quad \limsup_{d\to\infty}   \|\mS_*\|_{L^p(\R^d)}<\infty.
\end{equation}
Inequalities  \eqref{eq: dimfree} together with an interpolation argument will reduce \eqref{eq: diff Lp} and \eqref{eq: limno1} in Theorem \ref{thm: limit} to the following result. We remark that due to \eqref{eq: point est} in order to prove \eqref{eq: limno1} it would suffice to justify only the first estimate in \eqref{eq: diff L2 S-G} below.
\begin{thm}\label{thm_gaussian_sphere}
Let $d\ge 3.$ There exists a universal constant $C>0$ such that 
  \begin{equation}
 \label{eq: diff L2 S-B}
 \Vert(\mS-\mB)_*\Vert_{L^2(\mathbb{R}^d)}\leq C d^{-1/4}.
 \end{equation}
Moreover, for each $\alpha>0$ there exists a constant $C(\alpha)$ depending only on $\alpha>0$ and such that 
 \begin{equation}
 \label{eq: diff L2 S-G}
    \Vert  (\mS-\tmG)_*\Vert_{L^2(\mathbb{R}^d)}\leq C(\alpha) d^{-1/16+\alpha},\qquad \Vert(\tmG-\mB)_*\Vert_{L^2(\mathbb{R}^d)}\leq C(\alpha) d^{-1/16+\alpha}.
 \end{equation}
\end{thm}
\noindent As a corollary of \eqref{eq: diff L2 S-G}  and \eqref{eq: dimfree} combined with the known estimate \eqref{eq: mG p/(p-1)} for the operator norm of the Gaussian maximal function, we obtain information about the asymptotic behavior of the optimal constant for the Hardy-Littlewood maximal function and the spherical maximal function.
\begin{corollary}
\label{cor: S,B Lp}
   For each $p\in(1,\infty),$ there exists a dimension $d_0=d_0(p)$ and constants $C_p,\, \delta_p>0$, depending only on $p$ and such that for all $d\geq d_0$ we have 
\begin{equation*}
\|\mS_*\|_{L^p(\R^d)}\leq  \frac{p}{p-1} + C_p \frac{1}{d^{\delta_p}},\qquad \|\mB_*\|_{L^p(\R^d)}\leq  \frac{p}{p-1} + C_p \frac{1}{d^{\delta_p}}~.
\end{equation*}
\end{corollary}
\noindent This corollary motivates the upper bound for $\lm(p)$ in \eqref{eq: lm bounds}. The lower bound in \eqref{eq: lm bounds} will be a consequence of the following monotonicity result for the operator norm of the Gaussian maximal function.
\begin{thm}
    \label{propo:GausMonotonicity}
    For each $p\in(1,\infty)$ and $d\geq 1$ we have that
    \begin{align}\label{eq:gauss_norm_mono_intro}
        \|\mG_\ast^{d+1}\|_{L^p(\mathbb{R}^{d+1})}\geq \|\mG_\ast^{d}\|_{L^p(\mathbb{R}^{d})}~.
    \end{align}
    Moreover, there exist $0.4 <c\leq 1$ such that $\|\mG_\ast^{1}\|_{L^p(\mathbb{R})}\geq \max \left\lbrace c \frac{p}{p-1} , 1 \right\rbrace $.  In particular, $\|\mG_\ast^{1}\|_{L^p(\mathbb{R})}> \frac{2}{5} \frac{p}{p-1}>1$ for all $p\in (1,\frac{5}{3})$.
\end{thm}
\begin{remark}
\label{rem: TensMax}
We stress that the proof of \eqref{eq:gauss_norm_mono_intro} relies solely on the fact that the Gaussian kernel is positive, even, and of tensor product type.
In particular, the same monotonicity property holds for any maximal function whose convolution kernel is positive, even, and of tensor product type; an example of this is the centered maximal function associated with cubes in $\mathbb{R}^d$.
\end{remark}
\begin{remark}
\label{rem: LowHL}
Note that Theorem \ref{propo:GausMonotonicity} combined with \eqref{eq: point est} shows that we also have
\[\|\mB_\ast^{d}\|_{L^p(\mathbb{R}^{d})}\ge  \max \left\lbrace c \frac{p}{p-1} , 1 \right\rbrace\]
in all dimensions $d\ge 1$ and for all $p\in (1,\infty).$
\end{remark}

The main tools in the proof of Theorem \ref{thm_gaussian_sphere} are identities and estimates for the Fourier multiplier symbols associated, respectively, with the Gaussian, ball, and spherical averaging operators together with a general principle which is the content of Theorem \ref{general_thm} below. This principle states that under certain assumptions on the multiplier symbol, one can control the $L^2(\R^d)$ norm of a maximal function corresponding to Fourier multiplier operators by the $L^2(\R^d)$ norm of a single operator. 
\begin{thm}\label{general_thm}
    Let $a$ be a bounded, differentiable function on $\mathbb{R}^d$. Assume that there exists a constant $K>0$
such that the following inequalities
\begin{align}\label{eq_general_thm_pointwise_estimate_multiplier_1}
    |a(\xi)|\leq K \frac{|\xi|}{\sqrt{d}}  ~,
\qquad
    |a(\xi)|\leq K \frac{\sqrt{d}}{|\xi|} ~,
\end{align}
\begin{align}\label{eq_general_thm_pointwise_estimate_multiplier_gradient}
    \vert \langle \xi, \nabla a(\xi)\rangle \vert \leq K~
\end{align}
hold for almost every $\xi \in \R^d.$ Consider the family of Fourier multiplier operators $\lbrace \Fm_t\rbrace_{t>0}$ given by
    $$\widehat{\Fm_t f}(\xi)=a(t\xi)\widehat{f}(\xi)~.$$
   Then there exists a constant $C>0$ depending only on $K$ and such that 
    \begin{align*}
        \Vert  \Fm_*  \Vert_{L^2(\mathbb{R}^d)} \leq C \Vert a \Vert_{L^\infty(\mathbb{R}^d)}^{ 1/4}~. 
    \end{align*}
\end{thm}

\subsection{Historical background}

When $p=\infty,$ then clearly $ \|\mG_*\|_{L^p(\R^d)}=\|\mB_*\|_{L^p(\R^d)}= \|\mB_*\|_{L^p(\R^d)}=1.$ To our knowledge, for $p\in (1,\infty),$ the exact value of the $L^p$ norms of the centered maximal operators $ \|\mG_*\|_{L^p(\R^d)},$ $ \|\mB_*\|_{L^p(\R^d)},$ and $ \|\mS_*\|_{L^p(\R^d)},$ is not known for any $d\ge 1$. However, in dimension $d=1$ more information is available in the literature.  The best constant in the weak-type $(1,1)$ inequality for $\mB_*$ was established by Melas \cite{Me03} who showed that 
$$\Vert \mB_* \Vert_{L^1(\mathbb{R})\rightarrow L^{1,\infty}(\mathbb{R})}= \frac{11+\sqrt{61}}{12}~.$$ Moreover, Grafakos, Montgomery-Smith, and Motrunich \cite{GM-SM1} provided an explicit formula for the best constant $c_p,$ $p\in (1,\infty)$ in
\[
\|\mB_* f\|_{L^p(\R)}\le c_p \|f\|_{L^p(\R)},
\]
 when the input function $f$ is positive and convex except at one point. Additionally, the norm of the uncentered Hardy-Littlewood maximal operator, $\mB_\ast^u$, in dimension $d=1$ was obtained by Grafakos and Montgomery-Smith \cite{GM-S1}
  and it satisfies $\frac{p}{p-1}\leq \| \mB_\ast^u\|_{L^p(\mathbb{R})}\leq \frac{2p}{p-1}$.

Since obtaining explicit values for $\|\mG_*\|_{L^p(\R^d)},$ $\|\mB_*\|_{L^p(\R^d)},$ and $\|\mS_*\|_{L^p(\R^d)},$ seems a difficult task, one may instead try to obtain estimates from above or below. We first discuss explicit estimates from above.

It is well known that 
\begin{equation}
\label{eq: mG p/(p-1)}
\|\mG_*\|_{L^p(\R^d)}\le \frac{p}{p-1},\qquad p\in (1,\infty].
\end{equation}
What lies at the core of this inequality is the fact that $\mG_t$ is a symmetric-diffusion semigroup. Then one may apply Rota's dilation theorem together with Doob's martingale maximal inequality to conlcude \eqref{eq: mG p/(p-1)}, see e.g.\ \cite[Chapter IV, Section 4, p.\ 106]{St_topics}. We refer the reader to \cite[Equation~1.20.G*]{DGM18} for a sketch of different proof of \eqref{eq: mG p/(p-1)}. While the constant $p/(p-1)$ is optimal for Doob's maximal inequality it is unclear to us if $\|\mG_*\|_{L^p(\R^d)}=p/(p-1).$ 

The articles of Stein \cite{St82} and Stein and Strömberg \cite{SS83} imply the dimension-free bounds \eqref{eq: dimfree} for both $\|\mB_*\|_{L^p(\R^d)}$ and $\|\mS_*\|_{L^p(\R^d)}$, yet less is known about explicit estimates for these norms. In particular we do not know if \eqref{eq: mG p/(p-1)} holds for $\|\mB_*\|_{L^p(\R^d)}.$ 
In dimension one, the aforementioned result of Grafakos and Montgomery-Smith \cite{GM-S1} for the uncentered Hardy-Littlewood maximal operator combined with the pointwise estimate $ \mB_\ast f(x)\leq \mB_\ast^uf(x)$ gives
$$ \|\mB_\ast\|_{L^p(\mathbb{R})}\leq 2\frac{p}{p-1}~$$
for all $p>1$. Explicit upper bounds in arbitrary dimension have been already established by Stein and Strömberg \cite{SS83}, who showed that there exists a universal constant $C>0$ such that
\begin{equation}
\label{eq: mB StStr}
\|\mB_*\|_{L^p(\R^d)}\leq C\sqrt{d}\frac{p}{p-1}~,
\end{equation}
for all dimension $d$ and all $1<p\leq\infty$. 
In \cite{AC94}, Auscher and Carro have provided the following dimension-free improvement valid for $p\geq 2$ and  $d\ge 2$
\begin{equation}
\label{eq: mB AC}
\|\mB_*\|_{L^p(\R^d)}\leq (2+\sqrt{2})^{2/p}~.
\end{equation}
 Furthermore,
from Theorem 1.2, case $k=2$ in \cite{KWZ}, one can deduce that
\begin{equation}
\label{eq: mB KWZ}
\|\mB_*\|_{L^p(\R^d)}\leq C\left(\frac{p}{p-1}\right)^{3}\|f\|_{L^p(\R^d)}~,
\end{equation}
where $C>1$ is a universal constant. This is because for $k=2$ the factorization operator $M_k^t$ from \cite{KWZ} coincides with $\mB_t,$ see e.g.\ \cite[p.\ 427]{Ver1}.
When the input function $f$ is radially decreasing Aldaz and P\'erez-L\'azaro \cite[Corollary 2.8]{APL} established that
\begin{equation}
\label{eq: mB APL raddec}
\|\sup_{t>0}|\mB_t f|\|_{L^p(\R^d)}\leq 2\left(\frac{p}{p-1}\right)^{1/p}\|f\|_{L^p(\R^d)}~,
\end{equation}
for all $p\in (1,\infty].$ Moreover, from Men\'arguez and Soria \cite[Theorem 3]{MS1} it follows that \eqref{eq: mB APL raddec} remains true for general radial functions $f$ at the price of increasing the upfront constant from $2$ to $8.$ We note that the constants on the right-hand sides of \eqref{eq: mB StStr}, \eqref{eq: mB AC},  \eqref{eq: mB KWZ}, and \eqref{eq: mB APL raddec} above are strictly larger than $p/(p-1).$  

It is worth to add, that dimension-free estimates for $\|B_*\|_{L^p(\R^d)}$ have been vastly generalized to dimension-free estimates on $L^p(\R^d)$ for centered maximal functions associated with general symmetric convex bodies. These generalizations include contributions by Bourgain \cite{Bo1}, Carbery \cite{Car1} and M\"uller \cite{Mul1}. We refer the interested reader to the short survey article \cite{BMSW21} and to the comprehensive treatment in \cite{DGM18}. However, we note that estimates for general symmetric convex bodies are less explicit in terms of $p$ then those mentioned in the previous paragraph and thus they are less directly related to our discussion.

When it comes to the maximal spherical averages one may establish the following variant of \eqref{eq: mB AC} for $d\ge 3$ and $p\ge 2$
\begin{equation*}
\|\mS_*\|_{L^p(\R^d)}\leq C^{1/p},~
\end{equation*}
where $C>0$ is a universal constant.

Explicit estimates from below are even sparser.
In dimension one, the aforementioned result of Grafakos and Montgomery-Smith \cite{GM-S1} for the uncenterd Hardy-Littlewood maximal operator combined with the pointwise estimate $\frac{1}{2}\mB_\ast^uf(x)\leq \mB_\ast f(x)$ gives
$$\max \left\lbrace 1, \frac{1}{2}\frac{p}{(p-1)}\right\rbrace \leq \|\mB_\ast\|_{L^p(\mathbb{R})}~.$$
In \cite[Remark 2.9]{APL} Aldaz and P\'erez L\'azaro proved that for all $d\geq 1$ and $p>1$
\begin{equation}
\label{eq: mB bel APL}
\|\mB_{*}\|_{L^p(\R^d)}\ge \left(1+\frac{1}{2^{dp}(p-1)}\right)^{1/p}.
\end{equation}
Furthermore in \cite[Proposition]{CG1} and \cite[Theorem~3.2]{DSS} it has been shown that, for all $p\in (1,\infty)$ it holds
\begin{equation*}
\|\mB_{*}\|_{L^p(\R^d)}\ge b_{p,d},
\end{equation*}
where $b_{p,d}:=\mB_*(|x|^{-d/p})(1,0,\ldots,0)$.

\subsection{Structure of the paper and our methods}
Section \ref{sec: general_thm} is devoted to the proof of Theorem \ref{general_thm}. The proof is a variation of the analysis in \cite[Section~4]{BMSW21} and the reasoning is split into considering the dyadic maximal function and an $\ell^2$ sum of maximal functions over dyadic intervals. Compared to \cite{BMSW21}, our  variant of the proof has the crucial advantage that it allows us to control the operator norm of a maximal function by the operator norm of the associated Fourier multiplier operator. We also note that results with similar reasoning appeared previously in the literature, among others in \cite[Lemma~3]{B86}, \cite[Proposition i) p.\  271]{Car1}, \cite[Theorem A]{DRdF1}, and \cite[Section 3, Corollary]{RdF1}. It is likely that using these methods one can prove a variant of Theorem \ref{general_thm}. However, we found the current formulation and proof of Theorem \ref{general_thm} most convenient for our goals. We remark that the hypotheses \eqref{eq_general_thm_pointwise_estimate_multiplier_1} and \eqref{eq_general_thm_pointwise_estimate_multiplier_gradient} in the statement of Theorem \ref{general_thm} are not sharp. For example, the pointwise conditions in \eqref{eq_general_thm_pointwise_estimate_multiplier_1} can be replaced by
$$|a(\xi)|\leq C_1 \bigg(\frac{|\xi|}{\sqrt{d}}\bigg)^b,\qquad |a(\xi)|\le C_2\bigg( \frac{\sqrt{d}}{|\xi|}\bigg)^b~,$$
for some $b>0$. 
Theorem \ref{general_thm} can be 
  also reformulated in the spirit of the abstract approach in \cite{MSZK20}. In particular, one may prove an $r$-variational estimate instead of the maximal function estimate. However, the current formulation of Theorem \ref{general_thm}  will be enough for our purpose. 

Next, in Section \ref{sec: auxlem} we record some useful identities and estimates involving Bessel functions that will be needed in the upcoming sections.

In Section \ref{sec: pointmult} we justify that the pairwise differences of the Fourier multiplier symbols $\mu,$ $m$ and $g$, corresponding to the averages $\mS_1,$ $\mB_1,$ and $\tmG_1,$ respectively, 
satisfy the assumptions \eqref{eq_general_thm_pointwise_estimate_multiplier_1} and \eqref{eq_general_thm_pointwise_estimate_multiplier_gradient} of Theorem \ref{general_thm}. This is achieved with the help of estimates from Section \ref{sec: auxlem}.

In Section \ref{sec: proof main} we prove our main results - Theorem \ref{thm: limit} and \ref{thm_gaussian_sphere}. To establish Theorem \ref{thm_gaussian_sphere} we apply Theorem \ref{general_thm} with $\mA_t=\mS_t-\tmG_t$ and $\mA_t=\mS_t-\mB_t.$ The crucial point is to prove that $\|\mu-g\|_{L^{\infty}(\R^d)}$ and  $\|\mu-m\|_{L^{\infty}(\R^d)}$ are both $o(d)$ as $d\to \infty.$ This is the content of Proposition \ref{pro: g-mu, m-mu}. To prove that
$\|\mu-g\|_{L^{\infty}(\R^d)}=o(d)$ we express $\mu$ in terms of the expectation of the standard normal vector $X$, see Lemma \ref{lemma_identity_for_mu} and apply a concentration of measure result for $|X|^2$ from \cite{LaMa}. 
Then, we derive Theorem \ref{thm: limit} from Theorem \ref{thm_gaussian_sphere}. We finish Section \ref{sec: proof main} by highlighting implications of Theorem \ref{thm: limit} for a family of maximal operators connecting $\mB_\ast$ with $\mS_\ast$ considered by Dosidis and Grafakos in \cite{DG21}.

In Section \ref{sec:lowerBound} we prove Theorem \ref{propo:GausMonotonicity}. Namely, we rely on the tensor product structure of the Gaussian kernel to show that the operator norm $\|\mG_\ast\|_{L^p(\mathbb{R}^d)}$ is non-decreasing with the dimension. Then, we compute a lower bound for $\|\mG_\ast^1\|_{L^p(\mathbb{R})}$ (and, consequently, for $\|\mG_\ast^d\|_{L^p(\mathbb{R}^d)}$) by testing $\mG_\ast^1$ on a Gaussian function. 

 Finally, in Section \ref{sec:ConcludingRemarks} we discuss the norms on $L^p(\R^d)$ of the maximal functions applied to three natural radial inputs: a homogenous function, the characterstic function of the ball, and a Gaussian function.  We observe that, as $d\to \infty,$ all of these examples give the trivial conclusion $\lm(p)\ge 1$, leading to the following question.
 \begin{Que}
	\label{que: limrad}
    Fix $p\in (1,\infty)$ and let
$\|\mG_*\|_{L^p_{rad}(\R^d)},$ $\|\mB_*\|_{L^p_{rad}(\R^d)},$ $\|\mS_*\|_{L^p_{rad}(\R^d)},$ be the operator norms of the Gaussian, ball, and spherical maximal functions on $L^p(\R^d)$ restricted to radial functions.
Is it true that
\begin{equation}
 \label{eq: limrad 1}
    \lim_{d\to\infty}   \|\mG_*\|_{L^p_{rad}(\R^d)} =\lim_{d\to\infty}   \|\mB_*\|_{L^p_{rad}(\R^d)}=\lim_{d\to\infty}   \|\mS_*\|_{L^p_{rad}(\R^d)}=1\quad  ?
    \end{equation}		
\end{Que}

\subsection{Notation}
\begin{enumerate}
\item Throughout the paper the letter $d\in \N$ is reserved for the dimension and all inexplicit constants will be
independent of $d$.  
\item For two nonnegative quantities $X, Y$
we write $X \lesssim Y$ if there is an absolute universal constant
such that $X\le CY$. In particular, if the quantities $X$ and $Y$ involve the dimension $d$ and an additional parameter such as $\xi\in \R^d,$ then the universal constant $C$ in the estimate  $X\le CY$ is independent of these parameters. For instance, for a function $a\colon\R^d\to \C$
the estimate $|a(\xi)|\lesssim 1$ means that there is a universal constant $C>0$ such that $|a(\xi)|\le C$ for all dimensions $d$ and all $\xi \in \R^d.$

\item The Fourier transform of a function $f\in L^1(\mathbb{R}^d)$ is defined by the formula
$$\widehat{f}(\xi)=\mF(f)(\xi)=\int_{\mathbb{R}^d} f(x)e^{-2\pi i x\cdot\xi} \dd x,\qquad \xi\in\R^d~.$$
\item We denote by $P_t$ the Poisson semigroup defined for $f\in L^2$ by
\begin{align}\label{defi_Poisson_semigr}
\widehat{P_tf}(\xi)=p_t(\xi)\widehat{f}(\xi)~, \qquad  p_t(\xi)=e^{-2\pi t \tfrac{|\xi|}{\sqrt{d}}}~.
\end{align}
\item 
We will also need the resolution of the identity given by the Poisson projections
\begin{align}\label{defi_Poisson_projection}
f=\sum_{n\in\mathbb{Z}}S_nf~, \quad f\in L^2(\mathbb{R}^d)~, \qquad \text{where } \; S_n=P_{2^{n-1}}-P_{2^n}~.
\end{align}
\item 
We let $J_\nu,$ $\nu>0,$ be the Bessel function of the first kind of order $\nu.$

\item 
 We define the Fourier multipliers
 \begin{align*}
  \mu(\xi)  &=\frac{1}{\sigma(\mathbb{S}^{d-1})} \int_{\mathbb{S}^{d-1}} e^{- 2\pi i x\cdot \xi} \dd\sigma(x)= \frac{\Gamma(\tfrac{d}{2})}{\pi^{d/2-1}} |\xi|^{-\tfrac{d}{2}+1} J_{\tfrac{d}{2}-1}(2\pi |\xi|)~, \\
  m(\xi)  &=  \frac{1}{|B|}  \int_{B} e^{-2\pi i x\cdot \xi} \dd x = d \int_0^1 \mu(s \xi) s^{d-1} \dd s ~, \\
g(\xi)  &=e^{-\tfrac{2\pi^2|\xi|^2}{d}}~,
\end{align*}
associated, respectively, with the averaging operators $\mS_t$, $\mB_t$, $\tmG_t$, 
$$\mF(\mS_tf)(\xi)=\mu(t\xi)\widehat{f}(\xi)~, \quad \mF(\mB_tf)(\xi)=m(t\xi)\widehat{f}(\xi)~, \quad \mF(\tmG_t f)(\xi)=g(t\xi)\widehat{f}(\xi)~.$$

\end{enumerate}

\subsection*{Acknowledgments}
The authors are grateful to Tony Carbery for literature references. 

This research was funded in whole or in part by National Science Centre, Poland, research project 2022/46/E/ST1/00036. For the purpose of Open Access, the authors have applied a
CC-BY public copyright licence to any Author Accepted Manuscript (AAM) version arising from this submission.

\section{Proof of Theorem \ref{general_thm}}

\label{sec: general_thm}
We start by proving the following dyadic variant of Theorem \ref{general_thm}. 
\begin{proposition}\label{general_thm_dyadic}
    Let  $a=a_d$ be a bounded function on $\mathbb{R}^d$, $d\geq 1$. Assume that there exists a constant $K$ such that $a=a_d$ satisfies pointwise a.e. the inequalities in \eqref{eq_general_thm_pointwise_estimate_multiplier_1}.
    Consider the family of Fourier multiplier operators $\lbrace \Fm_t\rbrace_{t>0}$ given by
    $$\widehat{\Fm_t f}(\xi)=a(t\xi)\widehat{f}(\xi)~.$$
    Then for all $f\in L^2(\R^d)$ we have
    \begin{align*}
        \Vert \sup_{n\in\mathbb{Z}} |\Fm_{2^n}f| \Vert_{L^2(\mathbb{R}^d)} \leq 2K^{3/4} \Vert a \Vert_{L^\infty(\mathbb{R}^d)}^{1/4}\Vert f \Vert_{L^2(\mathbb{R}^d)}.
    \end{align*}
\end{proposition}

\proof Using Plancherel's identity, estimates in  \eqref{eq_general_thm_pointwise_estimate_multiplier_1}, and the following pointwise uniform estimate
$$\sum_{n\in\mathbb{Z}}\bigg(\min \bigg\lbrace \frac{2^n|\xi|}{\sqrt{d}}, \frac{\sqrt{d}}{2^n|\xi|}\bigg\rbrace\bigg)^{3/2}\le \sum_{n\in\mathbb{Z}}\min \bigg\lbrace \frac{2^n|\xi|}{\sqrt{d}}, \frac{\sqrt{d}}{2^n|\xi|}\bigg\rbrace  \leq 4~,$$
we have that
\begin{equation*}\begin{split}
    \bigg\Vert \sup_{n\in\mathbb{Z}} |\, \Fm_{2^n}f|\, \bigg\Vert_{L^2(\mathbb{R}^d)}^2 & \leq \sum_{n\in\mathbb{Z}} \Vert \Fm_{2^n}f\Vert_{L^2}^2 \\
    & = \sum_{n\in\mathbb{Z}} \int_{\mathbb{R}^d} |a(2^n\xi)|^2|\widehat{f}(\xi)|^2 \dd\xi \\
    & \leq  K^{3/2} \Vert a \Vert_{L^\infty(\mathbb{R}^d)}^{1/2} \sum_{n\in\mathbb{Z}} \int_{\mathbb{R}^d} \bigg(\min \bigg\lbrace \frac{2^n|\xi|}{\sqrt{d}}, \frac{\sqrt{d}}{2^n|\xi|}\bigg\rbrace\bigg)^{3/2}|\widehat{f}(\xi)|^2 \dd\xi \\
    & \leq 4 K^{3/2} \Vert a \Vert_{L^\infty(\mathbb{R}^d)}^{1/2} \Vert f \Vert_{L^2(\mathbb{R}^d)}^2~.
    \end{split}
\end{equation*}
\qed

To treat the non-dyadic version of the maximal operator we rely on the following decomposition
\begin{align}\label{decomposition_maximal_operator}
    \sup_{t>0} |\Fm_tf| \leq \sup_{n\in\mathbb{Z}} |\Fm_{2^n}f| + \bigg( \sum_{n\in\mathbb{Z}} \sup_{t\in [2^n,2^{n+1}]}|\Fm_t f- \Fm_{2^n}f|^2\bigg)^{1/2}~.
\end{align}
In view of Proposition \ref{general_thm_dyadic}, to complete the proof of Theorem \ref{general_thm}  it remains to estimate the second term in the right-hand-side of \eqref{decomposition_maximal_operator}.
To this end, we rely on the following Rademacher-Menshov type numerical inequality which can be found, for example, in \cite[Lemma~2.5]{MSZK20} and which holds for all $n\in\mathbb{Z}$ and for any continuous function $h:[2^{n},2^{n+1}]\rightarrow\mathbb{C}$
$$\sup_{t\in [2^n,2^{n+1}]}|h(t)-h(2^n)|\leq \sqrt{2} \sum_{\ell\in\mathbb{N}_0} \bigg( \sum_{m=0}^{2^\ell-1} |h(2^n+2^{n-\ell}(m+1))-h(2^n+2^{n-\ell}m)|^2 \bigg)^{1/2}~.$$
In view of this, and appealing to the decomposition $f=\sum_{j\in\mathbb{Z}}S_j f$ defined in \eqref{defi_Poisson_projection},  we have 
\begin{align*}
    \bigg\Vert \bigg( & \sum_{n\in\mathbb{Z}} \sup_{t\in [2^n,2^{n+1}]}|\Fm_tf  - \Fm_{2^n}f|^2\bigg)^{1/2} \bigg\Vert_{L^2(\mathbb{R}^d)} \\ 
    & \leq \sqrt{2} \sum_{\ell \geq 0} \sum_{j\in\mathbb{Z}} 
    \bigg\Vert \bigg( \sum_{n\in\mathbb{Z}} \sum_{m=0}^{2^\ell -1}\big|(\Fm_{2^n+2^{n-\ell}(m+1)}  - \Fm_{2^n+2^{n-\ell}m})S_{j+n}f \big|^2\bigg)^{1/2} \bigg\Vert_{L^2(\mathbb{R}^d)}~.
\end{align*}
We proceed by studying the norm on the right-hand-side of the last display, namely the quantity
\begin{align}\label{last_intermediate_estimate_thm_proof}
    \bigg\Vert \bigg( \sum_{n\in\mathbb{Z}} \sum_{m=0}^{2^\ell -1}  \big|(\Fm_{2^n+2^{n-\ell}(m+1)} & -  \Fm_{2^n+2^{n-\ell}m})S_{j+n}f \big|^2\bigg)^{1/2} \bigg\Vert_{L^2(\mathbb{R}^d)} ~.
\end{align}
We will obtain two estimates for \eqref{last_intermediate_estimate_thm_proof} and then we will interpolate between them. To derive the first estimate we need some preliminary bounds. First, in view of \eqref{eq_general_thm_pointwise_estimate_multiplier_1}, we have the bound
\begin{align*}
    |a((2^n + 2^{n-\ell}(m+1))\xi)-a((2^n+2^{n-\ell}m)\xi)|\leq 4 K\min \bigg\lbrace \frac{2^n|\xi|}{\sqrt{d}}, \frac{\sqrt{d}}{2^n|\xi|} \bigg\rbrace~.
\end{align*}
Moreover, in view of the definition of Poisson semigroup in \eqref{defi_Poisson_semigr} we have 
\begin{align*}
    |(e^{-2\pi 2^{n+j}|\xi|/\sqrt{d}}-e^{-2\pi 2^{n+j-1}|\xi|/\sqrt{d}})|\leq 8\pi \min \bigg\lbrace \frac{2^{n+j}|\xi|}{\sqrt{d}}, \frac{\sqrt{d}}{2^{n+j}|\xi|} \bigg\rbrace~.
\end{align*}
Combining the above two estimates we obtain that their product is bounded by $32\pi K2^{-|j|}$.
Using this fact together with Plancherel's identity and the estimates in \eqref{eq_general_thm_pointwise_estimate_multiplier_1} we obtain
\begin{align}\begin{split}\label{estimate_in_proof_last_intermediate_1}
    \bigg\Vert & \bigg( \sum_{n\in\mathbb{Z}} \sum_{m=0}^{2^\ell -1}  \big|(\Fm_{2^n+2^{n-\ell}(m+1)} - \Fm_{2^n+2^{n-\ell}m})S_{j+n}f \big|^2\bigg)^{1/2} \bigg\Vert_{L^2(\mathbb{R}^d)} \\
    & = \bigg( \sum_{n\in\mathbb{Z}} \sum_{m=0}^{2^\ell -1} \int_{\mathbb{R}^d} |a((2^n + 2^{n-\ell}(m+1))\xi)-a((2^n+2^{n-\ell}m)\xi)|^2 |(e^{-2\pi 2^{n+j}L|\xi|}-e^{-2\pi 2^{n+j-1}L|\xi|})|^2|\widehat{f}(\xi)|^2 \dd\xi \bigg)^{1/2}\\
    & \leq 8\sqrt{\pi}\sqrt{K} 2^{-|j|/2} \bigg( \sum_{n\in\mathbb{Z}} \sum_{m=0}^{2^\ell -1} \int_{\mathbb{R}^d} |a((2^n + 2^{n-\ell}(m+1))\xi)-a((2^n+2^{n-\ell}m)\xi)||\widehat{f}(\xi)|^2 \dd\xi \bigg)^{1/2}\\
    & \leq 16 \sqrt{\pi}K2^{-|j|/2} 2^{\ell/2} \bigg(  \int_{\mathbb{R}^d} \sum_{n\in\mathbb{Z}} \min \bigg\lbrace \frac{2^n|\xi|}{\sqrt{d}}, \frac{\sqrt{d}}{2^n|\xi|} \bigg\rbrace |\widehat{f}(\xi)|^2 \dd\xi \bigg)^{1/2}\\
    & \leq 32 \sqrt{\pi}K2^{-|j|/2} 2^{\ell/2}\Vert f \Vert_{L^2(\mathbb{R}^d)}~.
    \end{split}
\end{align}
To obtain the second estimate for \eqref{last_intermediate_estimate_thm_proof} we rely on the following inequality which follows from \eqref{eq_general_thm_pointwise_estimate_multiplier_gradient}  
\begin{align}\label{estimate_using_pointwise_gradient}
    |a((2^n + 2^{n-\ell}(m+1))\xi)-a((2^n+2^{n-\ell}m)\xi)| \leq \int_{2^n+2^{n-\ell}m}^{2^n+2^{n-\ell}(m+1)} |\langle t\xi, \nabla a(t\xi)\rangle| \frac{\dd t}{t} \leq K 2^{-\ell}~.
\end{align}
Therefore, using Plancherel's identity and inequality \eqref{estimate_using_pointwise_gradient} we obtain
\begin{align}
    \begin{split}\label{estimate_in_proof_last_intermediate_2}
 \bigg\Vert \bigg( \sum_{n\in\mathbb{Z}} \sum_{m=0}^{2^\ell -1} & \big|(\Fm_{2^n+2^{n-\ell}(m+1)}-  \Fm_{2^n+2^{n-\ell}m})S_{j+n}f \big|^2\bigg)^{1/2} \bigg\Vert_{L^2(\mathbb{R}^d)} \\
    & \leq  \bigg( \sum_{n\in\mathbb{Z}} \sum_{m=0}^{2^\ell -1} \big\Vert (\Fm_{2^n+2^{n-\ell}(m+1)} -  \Fm_{2^n+2^{n-\ell}m})S_{j+n}f \big\Vert_{L^2(\mathbb{R}^d)}^2\bigg)^{1/2}  \\
    & \leq \bigg( \sum_{n\in\mathbb{Z}} 2^{\ell} K^{(2-\alpha)} 2^{-\ell(2-\alpha)} 2^\alpha\Vert a \Vert_{L^\infty}^\alpha \Vert S_{n+j}f\Vert_{L^2}^2 \bigg)^{1/2} \\
    & \leq \sqrt{2K^{(2-\alpha)}} 2^{\ell(-1+\alpha)/2} \Vert a \Vert_{L^\infty}^{\alpha/2}\Vert f \Vert_{L^2(\mathbb{R}^d)}~,
    \end{split}
\end{align}
for any $0<\alpha<1$.
Combining \eqref{estimate_in_proof_last_intermediate_1}, \eqref{estimate_in_proof_last_intermediate_2} we obtain  
\begin{align*}
    \bigg\Vert \bigg( & \sum_{n\in\mathbb{Z}} \sup_{t\in [2^n,2^{n+1}]}|\Fm_tf  - \Fm_{2^n}f|^2\bigg)^{1/2} \bigg\Vert_{L^2(\mathbb{R}^d)} \\ 
    & \leq \sqrt{2} \sum_{\ell \geq 0} \sum_{j\in\mathbb{Z}} 
    \bigg\Vert \bigg( \sum_{n\in\mathbb{Z}} \sum_{m=0}^{2^\ell -1}\big|(\Fm_{2^n+2^{n-\ell}(m+1)}  - \Fm_{2^n+2^{n-\ell}m})S_{j+n}f \big|^2\bigg)^{1/2} \bigg\Vert_{L^2(\mathbb{R}^d)} \\
    & \leq \sqrt{2} (32 \sqrt{\pi}K)^\beta (\sqrt{2K^{(2-\alpha)}} \Vert a \Vert_{L^\infty}^{\alpha/2} )^{1-\beta}\sum_{\ell \geq 0} \sum_{j\in\mathbb{Z}} (2^{-|j|/2} 2^{\ell/2})^\beta (2^{\ell(-1+\alpha)/2})^{1-\beta} \Vert f \Vert_{L^2(\mathbb{R}^d)}~,
\end{align*}
where $0<\beta<1$ is such that $\beta < \tfrac{1-\alpha}{2-\alpha}$, ensuring summability of the double series in the last line.
Choosing appropriate numerical values for the parameters $\alpha$ and $\beta$ (e.g. $\alpha=7/12$ and $\beta=1/7$) concludes the proof of Theorem \ref{general_thm}.
 \qed

\section{Auxiliary lemmas for Bessel functions}
\label{sec: auxlem}
The purpose of this section is to collect some useful identities and estimates involving Bessel functions for later use. First, we recall the integral representation 
\begin{align}\label{integral_repr_bessel}
    J_\nu(r)= \frac{\big(\tfrac{r}{2}\big)^{\nu}}{\Gamma(\nu+\tfrac{1}{2})\sqrt{\pi}} \int_{-1}^1 e^{irs}(1-s^2)^{(2\nu-1)/2}
ds~, \quad \nu>-\frac{1}{2},\, r\geq 0~, \end{align}
see e.g. \cite[Appendix~B]{Gr14}.
Moreover, as a consequence of the recursion formulas 
\begin{align*}
   \frac{2\nu}{r} J_\nu(r) & = J_{\nu-1}(r)+J_{\nu+1}(r) \\\
   2 J'_\nu(r) & = J_{\nu-1}(r) - J_{\nu+1}(r)
\end{align*}
we have the identity
\begin{align}\label{first_prop_bessel}
J_\nu'(r)=J_{\nu-1}(r)-\frac{\nu}{r}J_\nu(r).
\end{align}

\begin{lemma}
Let $r>0$. The following identities hold,
\begin{align}\label{first_identity_rdmudr}
    r\frac{\dd}{\dd r} \mu(r)= 2 \frac{\Gamma(\tfrac{d}{2})}{\pi^{d/2-1}}\bigg(-\frac{d}{2}+1 \bigg) r^{-\tfrac{d}{2}+1}J_{\tfrac{d}{2}-1}(2\pi r) + 2\pi \frac{\Gamma(\tfrac{d}{2})}{\pi^{d/2-1}}r^{-\tfrac{d}{2}+2}J_{\tfrac{d}{2}-2}(2\pi r)~,
\end{align}
\begin{align}\label{second_identity_rdmudr}
     r\frac{\dd}{\dd r} \mu(r)= r (i2\sqrt{\pi}) \frac{\Gamma(\tfrac{d}{2})}{\Gamma(\tfrac{d}{2}-\tfrac{1}{2})} \int_{-1}^1 e^{i2\pi rs}s(1-s^2)^{(d-3)/2}\dd s~.
\end{align}
\end{lemma}
\proof Identity \eqref{first_identity_rdmudr} can be obtained by direct computation using the definition of $\mu$  together with identity \eqref{first_prop_bessel}. Identity \eqref{second_identity_rdmudr} can be obtained by direct computation using the definition of $\mu$ and the integral representation formula for Bessel functions \eqref{integral_repr_bessel}.
\qed

The following lemma provides a bound for the integral in the right-hand-side of \eqref{second_identity_rdmudr}.
\begin{lemma}
    There exists a constant $c>0$ independent of the dimension $d$ such that for all $d\geq 2$, $r\ge 0$, the following pointwise estimate holds
    \begin{align}\label{estimate_as_in_MSW}
       \bigg| \int_{-1}^1 e^{i2\pi rs}s(1-s^2)^{(d-3)/2}\dd s \bigg| \lesssim \frac{e^{-2\pi r/\sqrt{d}}}{d}+\frac{e^{-d/10}}{\sqrt{d}}~.
    \end{align}
\end{lemma}
\proof The statement is easily seen to be true for small values of $d$ hence we consider only sufficiently large values of $d\in\mathbb{N}$. We follow with very minor modifications the arguments in the proof of \cite[Lemma~4.1]{MSW24}, see also \cite[Lemma~3.6]{KW23}. We start by estimating
\begin{align*}
     &\bigg| \int_{-1}^1  e^{i2\pi rs}s(1-s^2)^{(d-3)/2}\dd s \bigg| = \frac{1}{\sqrt{d}}\bigg| \int_{-\sqrt{d}}^{\sqrt{d}} e^{i2\pi rs/ \sqrt{d}}\frac{s}{\sqrt{d}}(1-\frac{s^2}{d})^{(d-3)/2}\dd s \bigg|\\
          &\leq \frac{1}{\sqrt{d}}\bigg| \int_{\tfrac{\sqrt{d}}{2}\leq |s|\leq \sqrt{d}} e^{i2\pi rs/\sqrt{d}}\frac{s}{\sqrt{d}}\bigg(1-\frac{s^2}{d}\bigg)^{(d-3)/2}\dd s \bigg| + \frac{1}{\sqrt{d}}\bigg| \int_{|s|\leq \tfrac{\sqrt{d}}{2}} e^{i2\pi rs/ \sqrt{d}}\frac{s}{\sqrt{d}}\bigg(1-\frac{s^2}{d}\bigg)^{(d-3)/2}\dd s \bigg|~.
\end{align*}
Using the fact that $1-\frac{s^2}{d}\leq \frac{3}{4}$ for $|s|\geq \frac{\sqrt{d}}{2}$ we bound the first integral in the last display by
\begin{align*}
   \bigg| \int_{\tfrac{\sqrt{d}}{2}\leq |s|\leq \sqrt{d}} e^{i2\pi rs/\sqrt{d}}\frac{s}{\sqrt{d}}\bigg(1-\frac{s^2}{d}\bigg)^{(d-3)/2}\dd s \bigg|\leq \sqrt{d}\bigg( \frac{3}{4}\bigg)^{\tfrac{d-3}{2}}\lesssim e^{-d/10}~.
\end{align*}
To estimate the second integral
     \begin{align*}
      \bigg| \int_{|s|\leq \tfrac{\sqrt{d}}{2}} e^{i2\pi rs/ \sqrt{d}}\frac{s}{\sqrt{d}}\bigg(1-\frac{s^2}{d}\bigg)^{(d-3)/2}\dd s \bigg|~
      \end{align*}
we change the contour of integration. Let $\gamma:=\gamma_0\cup\gamma_1\cup\gamma_2\cup\gamma_3$ where
\begin{align*}
    &\gamma_0(s):=s \qquad &\text{for} \; s\in[-\tfrac{\sqrt{d}}{2},\tfrac{\sqrt{d}}{2}]~,\\
    &\gamma_1(s):=is+\frac{\sqrt{d}}{2} \qquad &\text{for} \; s\in[0,1]~,\\
    &\gamma_2(s):=-s+i \qquad &\text{for} \; s\in[-\tfrac{\sqrt{d}}{2},\tfrac{\sqrt{d}}{2}]~,\\
     &\gamma_3(s):=i(1-s)-\frac{\sqrt{d}}{2} \qquad &\text{for} \; s\in[0,1]~.
\end{align*}
As $z\mapsto e^{i2\pi rz/ \sqrt{d}}\tfrac{z}{\sqrt{d}}(1-\tfrac{z^2}{d})^{(d-3)/2}$ is holomorphic in $\lbrace z\in\mathbb{C}:\, |z|<\sqrt{d}\rbrace$, by Cauchy integral theorem we have
\begin{align*}
     \bigg| \int_{|s|\leq \tfrac{\sqrt{d}}{2}} e^{i2\pi rs/ \sqrt{d}}\frac{s}{\sqrt{d}}\bigg(1-\frac{s^2}{d}\bigg)^{(d-3)/2}\dd s \bigg| \leq & \sum_{k\in\lbrace 1,3 \rbrace} \bigg| \int_0^1 e^{i2\pi r\gamma_k(s)/ \sqrt{d}}\frac{\gamma_k(s)}{\sqrt{d}}\bigg(1-\frac{\gamma_k(s)^2}{d}\bigg)^{(d-3)/2}\dd s \bigg|\\
    & + \bigg| \int_{|s|\leq \tfrac{\sqrt{d}}{2}} e^{i2\pi r(-s+i)/ \sqrt{d}}\frac{(i-s)}{\sqrt{d}}\bigg(1-\frac{(i-s)^2}{d}\bigg)^{(d-3)/2}\dd s \bigg|~.  
\end{align*}
We bound the first term in the right-hand-side of the last display as
\begin{align*}
    \sum_{k\in\lbrace 1,3 \rbrace} \bigg| \int_0^1 e^{i2\pi r\gamma_k(s)/ \sqrt{d}}\frac{\gamma_k(s)}{\sqrt{d}}\bigg(1-\frac{\gamma_k(s)^2}{d}\bigg)^{(d-3)/2}\dd s \bigg|\leq 2 \bigg(\frac{1}{\sqrt{d}}+\frac{1}{2}\bigg)\bigg(\frac{3}{4}+\frac{1}{d}+\frac{1}{\sqrt{d}}\bigg)^{(d-3)/2}\lesssim e^{-d/10}~.
\end{align*}
As $e^{i2\pi r(-s+i)/ \sqrt{d}}=e^{-2\pi r/\sqrt{d}}e^{-i2\pi rs/\sqrt{d}}$, to conclude it is enough to show that
\begin{align}\label{estimate_last_term_after_contour}
     \int_{|s|\leq \tfrac{\sqrt{d}}{2}} |i-s|\bigg|1-\frac{(i-s)^2}{d}\bigg|^{(d-3)/2}\dd s \lesssim 1~.
\end{align}
Note that
\begin{align*}
   \bigg| 1- \frac{(s-i)^2}{d} \bigg|\leq 1+\frac{1-s^2}{d}+\frac{2|s|}{d} \leq \begin{cases}
 1+\frac{6}{d} & \text{for} \; |s|\leq \frac{5}{2}\\
 1-\frac{s^2}{36d} & \text{for} \; \frac{5}{2}\leq |s| \leq \frac{\sqrt{d}}{2}\\
    \end{cases} 
\end{align*}
and
\begin{align*}
 \int_{|s|\leq 1} |s|(1-s^2)^{(d-3)/2}\dd s= \frac{2}{d-1}~.
 \end{align*}
Using these facts, we see that the left-hand-side of \eqref{estimate_last_term_after_contour} is bounded by a universal constant times
 \begin{align*}
      & \int_{|s|\leq \tfrac{5}{2}} \bigg( 1+\frac{6}{d} \bigg)^{(d-3)/2}\dd s +  \int_{\tfrac{5}{2}<|s|\leq \tfrac{\sqrt{d}}{2}} |s|\bigg( 1-\frac{s^2}{36d} \bigg)^{(d-3)/2}\dd s  
      \lesssim  1+  \int_{|s|\leq \tfrac{\sqrt{d}}{2}} |s| \bigg(1-\frac{s^2}{36d}\bigg)^{(d-3)/2}\dd s \\
        \lesssim & 1+ d\int_{|s|\leq 1} |s|(1-s^2)^{(d-3)/2}\dd s \lesssim 1+   \frac{d}{d-1} 
        \lesssim  1~,
 \end{align*}
 hence concluding the proof.
 \qed 

\section{ Pointwise estimates for the multipliers - verification of hypothesis of Theorem \ref{general_thm}}
\label{sec: pointmult}

In order to apply Theorem \ref{general_thm} to our proof of Theorem \ref{thm_gaussian_sphere}, we need to verify that the multipliers $\mu-g$ and $\mu -m$ satisfy the hypotheses \eqref{eq_general_thm_pointwise_estimate_multiplier_1}, \eqref{eq_general_thm_pointwise_estimate_multiplier_gradient}. This is the content of the next propositions.
\begin{proposition}\label{propo_estimate_for_mu_minus_g}
For all $d\geq 3$, the multiplier $\mu-g$ satisfies the estimates 
\begin{equation*}
    |\mu(\xi)-g(\xi)|\lesssim\frac{|\xi|}{\sqrt{d}}~, \qquad |\mu(\xi)-g(\xi)|\lesssim \frac{\sqrt{d}}{|\xi|},\qquad 
    |\langle \xi, \nabla(\mu-g)(\xi) \rangle | \lesssim 1~,
\end{equation*}
for a.e.\ $\xi\in \R^d.$
\end{proposition}
\begin{proposition}\label{propo_estimates_mu_minus_m}
For all $d\ge 3$ the multiplier $\mu-m$ satisfies the estimates 
\begin{equation*}
    |\mu(\xi)-m(\xi)|\lesssim \frac{\sqrt{d}} {|\xi|} ~,
\qquad
    |\mu(\xi)-m(\xi)|\lesssim\frac{\sqrt{d}}{|\xi|},\qquad
    \vert \langle \xi, \nabla(\mu-m)(\xi)\rangle \vert \lesssim 1,
\end{equation*}
for a.e.\ $\xi\in \R^d.$
\end{proposition}

It is straightforward to see that $g(\xi)  =\exp\big(-\tfrac{2\pi^2|\xi|^2}{d}\big)$ satsfies
\begin{equation}
\label{eq: g est}
|g(\xi)-1|\lesssim \frac{|\xi|}{\sqrt{d}},\quad |g(\xi)|\lesssim \frac{\sqrt{d}}{|\xi|},\quad  |\langle \xi, \nabla g(\xi) \rangle | \lesssim 1. 
\end{equation}
Furthermore, it is well known that analogous estimates hold with $g$ replaced by $m$. This follows from \cite{B86}, for an explicit justification see e.g.\ \cite[Lemma 2.11]{BMSW20}. Therefore, to prove Propositions \ref{propo_estimate_for_mu_minus_g} and \ref{propo_estimates_mu_minus_m} it is enough to justify that \eqref{eq: g est} is also true when $g$ is replaced by $\mu.$ This is very close to the analysis in \cite[Section 4]{MSW24}. However, since the uniform estimate for the gradient $ |\langle \xi, \nabla \mu(\xi) \rangle | \lesssim 1$ is only mentioned in \cite{MSW24}  without proof, see \cite[eq.\ (5.20)]{MSW24}, we decided to provide a detailed proof of the next proposition. Note that,  in view of the discussion in this paragraph, Proposition \ref{propo_estimate_for_mu} implies both Propositions \ref{propo_estimate_for_mu_minus_g} and \ref{propo_estimates_mu_minus_m}.

\begin{proposition}\label{propo_estimate_for_mu}
For all $d\geq 3$, the multiplier  $\mu$ satisfies the estimates
\begin{equation*}
    |\mu(\xi)-1|\lesssim \frac{|\xi|}{\sqrt{d}}~, \qquad |\mu(\xi)|\lesssim \frac{\sqrt{d}}{|\xi|},\qquad
 |\langle \xi, \nabla \mu(\xi) \rangle | \lesssim 1,
\end{equation*}
for a.e.\ $\xi\in \R^d.$
\end{proposition}
\proof
We start by showing that
$$|\mu(\xi)-1|\leq 2\pi^2\frac{|\xi|}{\sqrt{d}}~.$$
We split the analysis into two cases. First, we consider the case $|\xi|\geq \sqrt{d}$. As $|\mu(\xi)|\leq 1$, we easily have that
$$|\mu(\xi)-1|\leq 2  \frac{|\xi|}{\sqrt{d}}~.$$
Let's now consider the case $|\xi|<\sqrt{d}$. It has been shown in \cite[Equation~4.9]{MSW24} that $|\mu(\xi)-1|\leq 2\pi^2 \big( \frac{|\xi|}{\sqrt{d}}\big)^{2}$. Combining this with the fact that, by hypothesis, $|\xi|<\sqrt{d}$ we obtain that
$$|\mu(\xi)-1|\leq  2\pi^2 \frac{|\xi|}{\sqrt{d}}~.$$

Next, we want to show that 
$$|\mu(\xi)| \lesssim \frac{\sqrt{d}}{|\xi|}.$$ 
We split the analysis into three cases. Assume first that $\sqrt{d}\geq |\xi|$. It has been shown in \cite[Equation~4.10]{MSW24} that  $|\mu(\xi)|\lesssim (\sqrt{d}/|\xi|)^{1/2}$. 
As by hypothesis $\sqrt{d}/|\xi|\geq 1$ it immediately follows that in this regime
$$|\mu(\xi)|\lesssim\frac{\sqrt{d}}{|\xi|}~.$$
Next we consider the case $\sqrt{d} < |\xi|<d$. It has been shown in \cite[Lemma~4.1]{MSW24} that there exists a constant $\alpha>0$ independent of the dimension such that for all $d\geq 2$ and all $\xi\in\mathbb{R}^d$ we have
$$|\mu(\xi)|\lesssim  e^{-2\pi|\xi|/\sqrt{d}}+e^{-\alpha d}.$$
Using this estimate, the fact that $e^{-|x|}\leq \tfrac{1}{|x|}$, and the hypothesis $1 < d/|\xi|$ we have 
\begin{align*}
    |\mu(\xi)|& \lesssim e^{-2\pi|\xi|/\sqrt{d}}+e^{-\alpha d}  \\
    & \lesssim \frac{\sqrt{d}}{|\xi|} +\frac{1}{\alpha\sqrt{d}}  \lesssim   \frac{\sqrt{d}}{|\xi|} +\frac{1}{\alpha\sqrt{d}}\frac{d}{|\xi|} \\
    & \lesssim \frac{\sqrt{d}}{|\xi|}~.
\end{align*}
We are left to consider the case $|\xi|\geq d$. Here we use the pointwise bound 
\begin{align}\label{eq_bessel_leq_r_power_minus_one_half}
|J_\nu(r)|\leq r^{-1/2}
\end{align}
which holds for $\nu\geq \tfrac{1}{2}$ and $r\geq 2\nu$, see e.g. \cite[Theorem~3]{Kr14}. This shows that for $d\ge 3$ we have
\begin{align*}
    |\mu(\xi)| & =\frac{\Gamma(\tfrac{d}{2})}{\pi^{d/2-1}}|\xi|^{-\tfrac{d}{2}+1} |J_{\tfrac{d}{2}-1}(2\pi|\xi|)| \leq \frac{\Gamma(\tfrac{d}{2})}{\pi^{d/2-1}}|\xi|^{-\tfrac{d}{2}+\tfrac{1}{2}} \\
    & \leq \frac{\Gamma(\tfrac{d}{2})}{\pi^{d/2-1}}d^{-\tfrac{d}{2}+\tfrac{3}{2}}|\xi|^{-1} \lesssim\frac{1}{|\xi|}\\
    & \lesssim\frac{\sqrt{d}}{|\xi|},
\end{align*}
where in the third inequality we used Stirling's approximation for the Gamma function 
\begin{equation}
\label{eq: Stirling}
\Gamma(s) = \sqrt{2 \pi /s} (s / e)^{s} (1 + O(s^{-1})),\qquad s \to \infty.
\end{equation}

Finally, we are left to show that $|\langle \xi, \nabla \mu(\xi) \rangle|\lesssim 1$ for all $\xi\in\mathbb{R}^d$. 
We split the analysis into two cases. First, we consider the case $r=|\xi|\geq d$. Using identity \eqref{first_identity_rdmudr} together with \eqref{eq_bessel_leq_r_power_minus_one_half} and \eqref{eq: Stirling} we have that for $d\geq 3$ 
\begin{align*}
   \bigg| r\frac{\dd}{\dd r} \mu(r)\bigg|& \leq 2 \frac{\Gamma(\tfrac{d}{2})}{\pi^{d/2-1}}\big(\tfrac{d}{2}-1 \big) r^{-\tfrac{d}{2}+1}\big| J_{\tfrac{d}{2}-1}(2\pi r)\big| + 2\pi \frac{\Gamma(\tfrac{d}{2})}{\pi^{d/2-1}}r^{-\tfrac{d}{2}+2}\big|J_{\tfrac{d}{2}-2}(2\pi r)\big|\\
   & \leq 2 \frac{\Gamma(\tfrac{d}{2})}{\pi^{d/2-1}}\big(\tfrac{d}{2}-1 \big) r^{-\tfrac{d}{2}+\tfrac{1}{2}} + 2\pi \frac{\Gamma(\tfrac{d}{2})}{\pi^{d/2-1}}r^{-\tfrac{d}{2}+\tfrac{3}{2}}\\
   & \leq 2 \frac{\Gamma(\tfrac{d}{2})}{\pi^{d/2-1}}\big(\tfrac{d}{2}-1 \big) d^{-\tfrac{d}{2}+\tfrac{1}{2}} + 2\pi \frac{\Gamma(\tfrac{d}{2})}{\pi^{d/2-1}}d^{-\tfrac{d}{2}+\tfrac{3}{2}}\\
   & \lesssim 1~.
\end{align*}
Next, we consider the case $r=|\xi|\leq d$. This can be verified as outlined in \cite[Remark~5.1]{MSW24}, we include the arguments for completeness. Using identity \eqref{second_identity_rdmudr} together with the estimate \eqref{estimate_as_in_MSW} we obtain 
\begin{align*}
    \bigg| r\frac{\dd}{\dd r} \mu(r)\bigg|& \leq r (2\sqrt{\pi}) \frac{\Gamma(\tfrac{d}{2})}{\Gamma(\tfrac{d}{2}-\tfrac{1}{2})}  \bigg| \int_{-1}^1 e^{i2\pi rs}s(1-s^2)^{(d-3)/2}\dd s \bigg| \\
    & \lesssim  r \frac{\Gamma(\tfrac{d}{2})}{\Gamma(\tfrac{d}{2}-\tfrac{1}{2})}  \bigg(\frac{e^{-2\pi r/\sqrt{d}}}{d}+\frac{e^{-cd}}{\sqrt{d}}\bigg) \\
    & \lesssim r \bigg(\frac{e^{-2\pi r/\sqrt{d}}}{\sqrt{d}}+e^{-d/10}\bigg)\lesssim 1~.
\end{align*}
\qed

\section{ Proofs of the main results - Theorem \ref{thm: limit} and \ref{thm_gaussian_sphere}}
\label{sec: proof main}

In this section we first prove Theorem \ref{thm_gaussian_sphere} and then deduce from it Theorem \ref{thm: limit}.

\subsection{Proof of Theorem \ref{thm_gaussian_sphere}.}

Let $\mA_t,$ $t>0,$ denote the operator corresponding to any of the differences $\mS_t-\tmG_t,$ $\mS_t-\mB_t$ or $\tmG_t-\mB_t$, let $a$ be the multiplier symbol of $\mA_1$ and let $\mA_*$ be the corresponding maximal function. In the previous section we established that $a$ satisfies the assumptions of Theorem \ref{general_thm} with a universal constant $K\lesssim 1.$ Hence, applying this theorem we see that
\begin{equation*}
\|\mA_*\|_{L^2(\R^d)}\lesssim \|a\|_{L^{\infty}(\R^d)}^{1/4}.
\end{equation*}
Thus, in order to prove  Theorem \ref{thm_gaussian_sphere} it remains to justify that $\|a\|_{L^{\infty}(\R^d)}\lesssim d^{-\delta}$ for $a$ being any two among the differences of $g-\mu,$ $\mu-m$ and $g-m$ and $\delta>0$ being a universal constant. This is the content of the next proposition.

\begin{proposition}
 \label{pro: g-mu, m-mu}
For all $d\ge 3$ we have
         \begin{equation}
       \label{eq: m-mu}
       |m(\xi)-\mu(\xi)|\lesssim d^{-1},\qquad \xi\in \R^d.
       \end{equation}
Furthermore, for each $\alpha>0$ and all $d\ge 3$ it holds
       \begin{equation}
       \label{eq: g-mu}
       |g(\xi)-\mu(\xi)|\leq C(\alpha) d^{-1/4+\alpha},\qquad \xi \in \R^d,
       \end{equation}
       where the constant $C(\alpha)$ depends only on $\alpha.$
\end{proposition}

Our proof of \eqref{eq: g-mu} relies crucially on two identities - one for $\mu$ and one for $g$ - which are the content of the following lemma. The proof of Lemma \ref{lem: identity_for_mu,g}  is a standard computation but we provide it for reader's convenience. Below $\mathcal{N}(0,\textrm{I}_d)$ stands for the multivariate normal distribution with identity $\textrm{I}_d$ covariance matrix.
\begin{lemma}\label{lem: identity_for_mu,g} The following pointwise identities hold for $\xi \in \R^d$
\begin{align}
\label{lemma_identity_for_g}
g(\xi)&=\mathop{\mathbb{E}}_{X\sim \mathcal{N}(0,\textrm{I}_d)} e^{- 2\pi i \tfrac{X}{\sqrt{d}}\cdot \xi},\\
\label{lemma_identity_for_mu}
   \mu(\xi) &= \mathop{\mathbb{E}}_{X\sim \mathcal{N}(0,\textrm{I}_d)} e^{-2\pi i \tfrac{X}{|X|}\cdot \xi}~.
\end{align}
\end{lemma}
\begin{proof} 
First we justify \eqref{lemma_identity_for_g}. Using the fact that the Fourier transform of a Gaussian function is itself a Gaussian function we have that
$$\mathop{\mathbb{E}}_{X\sim \mathcal{N}(0,\textrm{I}_d)} e^{- 2\pi i \tfrac{X}{\sqrt{d}}\cdot \xi}= \frac{1}{(2\pi)^{d/2}} \int_{\mathbb{R}^d} e^{-2\pi i \tfrac{x}{\sqrt{d}}\cdot\xi}e^{-|x|^2/2}dx = e^{-\tfrac{2\pi^2|\xi|^2}{d}}= g(\xi)~.$$

To prove \eqref{lemma_identity_for_mu} we compute
\begin{align*} \mathop{\mathbb{E}}_{X\sim \mathcal{N}(0,\textrm{I}_d)} e^{-2\pi i \tfrac{X}{|X|}\cdot \xi} & = \frac{1}{(2\pi)^{d/2}}\int_{\mathbb{R}^d}  e^{-2\pi i \tfrac{x}{|x|}\cdot \xi} e^{-\tfrac{|x|^2}{2}}\dd x \\
& = \frac{1}{(2\pi)^{d/2}}\int_0^\infty \bigg( \int_{\mathbb{S}^{d-1}}  e^{-2\pi i\omega \cdot \xi} \dd\sigma(\omega)\bigg) e^{-\tfrac{r^2}{2}} r^{d-1}\dd r  \\
&= \sigma(\mathbb{S}^{d-1}) \mu(\xi) \frac{1}{(2\pi)^{d/2}}\int_0^\infty e^{-\tfrac{r^2}{2}} r^{d-1}\dd r \\
&= \mu(\xi) \frac{1}{(2\pi)^{d/2}}\int_{\mathbb{R}^d} e^{-\tfrac{|x|^2}{2}}\dd x \\
&= \mu(\xi)~.
\end{align*}
\end{proof}

We have now all the tools to proceed with the proof of Proposition \ref{pro: g-mu, m-mu}.

\begin{proof}[Proof of Proposition \ref{pro: g-mu, m-mu}.]
We start with the the proof of \eqref{eq: m-mu}. Note that
 \[
\mu(\xi)-m(\xi)=\int_0^1 \bigg( \frac{\dd}{\dd s} \mu(s\xi) \bigg)s^d \dd s=\int_0^1 \langle s\xi, (\nabla \mu)(s\xi)\rangle s^{d-1} \dd s.
 \]
 Thus, using the third inequality in Proposition \ref{propo_estimate_for_mu} we reach
 \[
 |\mu(\xi)-m(\xi)|\lesssim \int_0^1 s^{d-1}\,ds \lesssim d^{-1}
 \]
which is \eqref{eq: m-mu}. 

We now move on the the proof of \eqref{eq: g-mu}. Here the analysis is split into two cases. First we assume that $|\xi|\geq d^{1/2+\varepsilon}$ for some $\varepsilon>0$ to be chosen later. In view of the second inequality in \eqref{propo_estimate_for_mu}  we have 
\begin{equation}
\label{eq: mu-g large xi}
|\mu(\xi)-g(\xi)|\leq |\mu(\xi)|+|g(\xi)|\lesssim \frac{d^{1/2}}{|\xi|}\leq d^{-\varepsilon}~.\end{equation}

Next, we consider the complementary case $|\xi| < d^{\tfrac{1}{2}+\varepsilon}$. In view of Lemma \ref{lem: identity_for_mu,g}, the following identity holds 
\begin{align*}
    |\mu(\xi)-g(\xi)|= \bigg|\mathop{\mathbb{E}}_{X\sim \mathcal{N}(0,\textrm{I}_d)}\bigg(e^{-2\pi i \tfrac{X}{|X|}\cdot \xi} - e^{- 2\pi i \tfrac{X}{\sqrt{d}}\cdot \xi} \bigg) \bigg|~.
\end{align*}
We further split the expectation above over the two complementary events $||X|^2-d|\leq 4d^{1/2+\alpha}$ and $||X|^2-d| > 4d^{1/2+\alpha}$, for some $\alpha\in(0,1/2)$ to be chosen later. To deal with the second case we rely on the following concentration inequality 
\begin{equation}
\label{eq: concentration}
\mathbb{P}(|X|^2\in [d-2d^{1/2+\alpha}, d+ 2d^{1/2+\alpha} + 2d^{\alpha}])\geq 1- e^{-d^\alpha}.\end{equation}
This inequality is a consequence of Lemma 1 and the comment following it in \cite{LaMa} with $a_i=1,$ $i=1,\ldots,D,$ and $x=d^{\alpha}$. Using \eqref{eq: concentration} we obtain
$$\bigg|\mathop{\mathbb{E}}_{X\sim \mathcal{N}(0,\textrm{I}_d)} \mathbf{1}_{\lbrace ||X|^2-d|\, \geq \, 4d^{1/2+\alpha} \rbrace}\big(e^{-2\pi i \tfrac{X}{|X|}\cdot \xi} - e^{- 2\pi i \tfrac{X}{\sqrt{d}}\cdot \xi} \big) \bigg| \leq 2 e^{-d^\alpha} ~.$$
Hence we are left to bound 
\begin{align*}
    \bigg|\mathop{\mathbb{E}}_{X\sim \mathcal{N}(0,\textrm{I}_d)} & \mathbf{1}_{\lbrace ||X|^2-d|\, \leq \, 4d^{1/2+\alpha} \rbrace}\big(e^{-2\pi i \tfrac{X}{|X|}\cdot \xi} - e^{- 2\pi i \tfrac{X}{\sqrt{d}}\cdot \xi} \big) \bigg| \\
    & \lesssim \mathop{\mathbb{E}}_{X\sim \mathcal{N}(0,\textrm{I}_d)} \mathbf{1}_{\lbrace ||X|^2-d|\, \leq \, 4d^{1/2+\alpha} \rbrace} \big|\tfrac{1}{|X|}-\tfrac{1}{\sqrt{d}}\big| |X\cdot\xi| \\
    &= \mathop{\mathbb{E}}_{X\sim \mathcal{N}(0,\textrm{I}_d)} \mathbf{1}_{\lbrace ||X|^2-d|\, \leq \, 4d^{1/2+\alpha} \rbrace} \tfrac{||X|^2-d|}{|X|^2\sqrt{d}+|X|d} |X\cdot\xi| \\
    & \lesssim \mathop{\mathbb{E}}_{X\sim \mathcal{N}(0,\textrm{I}_d)}\tfrac{d^{1/2+\alpha}}{d^{3/2}} |X\cdot\xi| \\
    & =d^{-1+\alpha}|\xi| \mathop{\mathbb{E}}_{{X_1} \sim \mathcal{N}(0,1)} |{X_1}|\\
    & \lesssim d^{-1/2+\alpha+\varepsilon}
\end{align*}
 where we have used invariance under rotation of $X$ to pass to the second equality and the fact that, by hypothesis, $|\xi| < d^{\tfrac{1}{2}+\varepsilon}$ to pass to the last line. Therefore, we justified that in the case $|\xi|<d^{1/2+\varepsilon}$ it holds
 \begin{equation*}
|\mu(\xi)-g(\xi)|\lesssim e^{-d^{\alpha}}+d^{-1/2+\alpha+\varepsilon}~.\end{equation*} 
In view of the above inequality and \eqref{eq: mu-g large xi}, choosing $\varepsilon=1/4$ and $\alpha\to 0^+$ we obtain \eqref{eq: g-mu}.

This completes the proof of Proposition \ref{pro: g-mu, m-mu}.
 \end{proof}

 \subsection{Proof of Theorem \ref{thm: limit}.} We justify \eqref{eq: diff Lp} first.
Take $q\in (1,\infty)$ and let $d_0=\lfloor 1+1/(q-1)\rfloor+1.$ Recalling \eqref{eq: dimfree} we see that
\begin{equation}
\label{eq: dim free three}
 \sup_{d \ge d_0 } \|(\mS-\tmG)_*\|_{L^q(\R^d)}<\infty,\qquad  \sup_{d \ge d_0 } \|(\mS-\mB)_*\|_{L^q(\R^d)} < \infty,\qquad \sup_{d\ge 1} \|(\tmG-\mB)_*\|_{L^q(\R^d)}<\infty.
\end{equation}
Thus, Marcinkiewicz interpolation theorem \cite[Theorem 1.3.2]{Gr14}, Theorem \ref{thm_gaussian_sphere} and \eqref{eq: dim free three} imply \eqref{eq: diff Lp}.

In remains to demonstrate \eqref{eq: limno1}. It is well known that for each $p>d/(d-1)$ we have $\|\mS_*\|_{L^p(\R^d)}\le \|\mS_*\|_{L^p(\R^{d-1})},$ see e.g.\ \cite[eq.\ (2.10)]{BMSW21}. In particular, for each $p\in (1,\infty)$ the limit $\lim_{d\to \infty}\|\mS_*\|_{L^p(\R^d)}$ 
exists. Thus, in view of \eqref{eq: diff Lp} the remaining two limits in \eqref{eq: limno1} also exist and are all equal. Finally, \eqref{eq: mG p/(p-1)} implies that these limits do not exceed $p/(p-1)$ while the lower bound in \eqref{eq: lm bounds} is justified by Theorem \ref{propo:GausMonotonicity}.
\qed

\begin{remark}
\label{rem: DG}
    A family of $m-$parameters maximal operators $\mathcal{S}^m_{\alpha,\ast}$, with $0\leq \alpha <1$, connecting the ($m-$parameters) Hardy-Littlewood maximal function with the ($m-$parameters) spherical maximal function has been studied by Dosidis and Grafakos in \cite{DG21}. Here, we focus on the one-parameter case $\mathcal{S}_{\alpha,\ast}=\mathcal{S}^1_{\alpha,\ast}$, 
    $$\mathcal{S}_{\alpha,\ast}(f)(x):=\sup_{t>0} \frac{2}{\omega_{d-1}\mathrm{B}(\frac{d}{2},1-\alpha)}\int_{B}|f(x-ty)|(1-|y|)^{-\alpha}\dd y~,$$
    where $\omega_{d-1}$ is the surface area of the unit sphere in $\mathbb{R}^d$, and $\mathrm{B}(a,b)$ is the beta function $\mathrm{B}(a,b):=\int_0^1 t^{a-1}(1-t)^{b-1}\dd t$. It has been shown in \cite{DG21} that, for any $f\in L^1_{loc}(\mathbb{R}^d)$ and $x\in\mathbb{R}^d$, the pointwise estimate
    $$\mB_\ast f(x)\leq S_{\alpha,\ast}f(x)\leq \mS_\ast f(x)$$
    holds.  In view of this fact and as a consequence of our Theorem \ref{thm: limit}, we immediately obtain that also for this family of maximal operators
    $$\lim_{d\rightarrow\infty} \|\mathcal{S}_{\alpha,\ast}\|_{L^p(\mathbb{R}^d)} = \lm(p)~.$$
    
\end{remark}

\section{Lower bounds for $\lm(p)$ - Proof of Theorem \ref{propo:GausMonotonicity}}\label{sec:lowerBound}
We split the proof of Theorem \ref{propo:GausMonotonicity} into two parts. First, we focus on the monotonicity property for $\|\mG_\ast\|_{L^p(\mathbb{R}^d)}$. 

\begin{proposition}\label{propo2:GausMonotonicity}
    For all $p\in(1,\infty)$ and $d\geq 1$ it holds that
    $$\|\mathcal{G}_\ast\|_{L^p(\mathbb{R}^{d+1})}\geq \| \mathcal{G}_\ast\|_{L^p(\mathbb{R}^d)}~.$$
\end{proposition}
\proof We start by observing that, by the monotone convergence theorem, it suffices to consider $0<t<T$ for some fixed $T>0$. 
Let $f\in L^p(\mathbb{R}^d)$, $\varphi\in L^1_{loc}$. We have that
\begin{equation*}
    \mG^{d+1}_t(f\otimes\varphi)(x,s)=\mG_t^df(x)\,\mG_t^1\varphi (s)~,
\end{equation*}
where $(x,s)\in\mathbb{R}^{d+1}$, $x\in\mathbb{R}^d$, $s\in\mathbb{R}$. We set \begin{equation}
\label{eq: phi,phi_N}
\varphi(s)=|s|^{-1/p}\qquad \textrm{and}\qquad  \varphi_N(s):=\varphi(s)\mathbf{1}_{[1,N]}(|s|).
\end{equation}Here $N>2$ is a large parameter which will be taken to infinity in the proof. Note that $\varphi_N, \, \varphi_{\frac{N}{2}}\in L^p(\mathbb{R})$.
Since the kernel of $\mG_t^1$ is even a change of variable gives
\begin{align}
\begin{split}\label{eq:convexity_low_b}
\mG_t^1\varphi_N (s) & = \frac{1}{\sqrt{4\pi t}} \int_{-\infty}^\infty \varphi_N(s-y) e^{-y^2/(4t)}\dd y\\
& = \frac{1}{\sqrt{4\pi }} \int_0^\infty \left( \varphi_N(s-\sqrt{t}y)+\varphi_N(s+\sqrt{t}y)\right) e^{-y^2/4}\dd y\\
& \geq \frac{2}{\sqrt{4\pi }} \int_0^M  \frac{\varphi_N(s-\sqrt{t}y)+\varphi_N(s+\sqrt{t}y)}{2} e^{-y^2/4}\dd y~,
\end{split}
\end{align}
for any $M>0$. At this point, we observe that $\varphi|_{(-\infty,0)}$ and $\varphi|_{(0,\infty)}$ are convex. Moreover, if $|s|>\sqrt{t}y$ then $\text{sign}(s+\sqrt{t}y)=\text{sign}(s-\sqrt{t}y)$. In particular, this is guaranteed for $|s|>10\sqrt{T}M$. Observe also that if $10\sqrt{T}M<|s|<\tfrac{N}{2}$ then both $|s+\sqrt{t}y|$ and $|s-\sqrt{t}y|$ lie between $1$ and $N$. Obviously, here we take $N>20\sqrt{T}M.$ Therefore, in such a range for $s$ using the convexity of $\varphi$ we obtain
$$\frac{\varphi_N(s-\sqrt{t}y)+\varphi_N(s+\sqrt{t}y)}{2}=\frac{\varphi(s-\sqrt{t}y)+\varphi(s+\sqrt{t}y)}{2}\geq \varphi(s)~. $$
Plugging this into \eqref{eq:convexity_low_b} we have that for all $0<t<T$
$$\mG_t^1\varphi_N (s) \geq \varphi(s)\mathbf{1}_{[ 10\sqrt{T}M,N/2]}(|s|) \frac{2}{\sqrt{4\pi}}\int_0^Me^{-y^2/4}\dd y\geq \varphi_{\frac{N}{2}}(s)\mathbf{1}_{[ 10\sqrt{T}M,\infty)}(|s|) I(M)~, $$
where $I(M):= \frac{2}{\sqrt{4\pi}}\int_0^Me^{-y^2/4}\dd y$ satisfies $\lim_{M\rightarrow\infty}I(M)=1$. A simple computation gives
\begin{align*}\|\varphi_N\|_{L^p(\mathbb{R})}&= 2^{1/p}(\log(N))^{1/p}, \\ \|\varphi_{\frac{N}{2}}\mathbf{1}_{(-\infty,-10\sqrt{T}M]\cup[ 10\sqrt{T}M,\infty]}\|_{L^p(\mathbb{R})}&= 2^{1/p}(\log(N)-\log(2)-\log(10\sqrt{T}M))^{1/p}~.\end{align*}
Combining all of these together, we have justified that
\begin{align*}
   & \frac{\|\sup_{0<t<T}\mG_t^{d+1}(f \otimes \varphi_N)\|_{L^p(\mathbb{R}^{d+1})}}{\|f\otimes \varphi_N\|_{L^p(\mathbb{R}^{d+1})}}  = 
    \frac{\|\sup_{0<t<T}\mG_t^{d+1}(f\otimes \varphi_N)\|_{L^p(\mathbb{R}^{d+1})}}{\|f\|_{L^p(\mathbb{R}^d)} \| \varphi_N\|_{L^p(\mathbb{R})}}\\
    & \qquad \geq  \frac{\|\sup_{0<t<T} \mG_t^{d}f\|_{L^p(\mathbb{R}^{d})}}{\|f\|_{L^p(\mathbb{R}^d)} } \left( \frac{I(M)2^{1/p}(\log(N)-\log(2)-\log(10\sqrt{T}M))^{1/p}}{2^{1/p}(\log(N))^{1/p}}\right)~.
\end{align*}
Taking the limit as $N\rightarrow\infty$ we obtain
$$\lim_{N\rightarrow\infty} \frac{\|\sup_{0<t<T}\mG_t^{d+1}(f \otimes \varphi_N)\|_{L^p(\mathbb{R}^{d+1})}}{\|f\otimes \varphi_N\|_{L^p(\mathbb{R}^{d+1})}} \geq \frac{\|\sup_{0<t<T}\mG_t^{d}f\|_{L^p(\mathbb{R}^{d})}}{\|f\|_{L^p(\mathbb{R}^d)} } I(M)~. $$
Finally, we let $M\rightarrow \infty$ and then we conclude by taking the limit for $T\rightarrow\infty$.
\qed 

Next, we compute lower bounds on $\|\mG_\ast\|_{L^p(\mathbb{R})}$ by testing $\mG_\ast$ on two natural examples, a homogeneous radial function and a Gaussian function.

\begin{lemma}
 For all $p\in (1,\infty)$ we have $\| \mathcal{G}_\ast \|_{L^p(\mathbb{R})} \geq \mG_\ast(|x|^{-1/p})(1) $. In particular
    $$\| \mathcal{G}_\ast \|_{L^p(\mathbb{R})} \geq \frac{2^{(p-1)/p}}{\sqrt{2e\pi}} \frac{p}{p-1}~.$$   
\end{lemma}
\proof It is easy to see that for $\varphi$ given by \eqref{eq: phi,phi_N} we have
$\mG_\ast(\varphi)(x)=\mG_\ast(\varphi)(1)|x|^{-1/p}$, $x\in \R.$ By standard approximation arguments (see e.g. \cite[Proof~of~Thm.~3.2]{DSS}), one also gets that $$\lim_{N\rightarrow\infty}\|\mG_\ast \varphi_N\|_{L^p(\mathbb{R})}/\|\varphi_N\|_{L^p(\mathbb{R})}\geq  \mG_\ast(|x|^{-1/p})(1)~.$$ 
A simple lower bound for $ \mG_\ast(|x|^{-1/p})(1)$ is 
\begin{align*}
 \mG_\ast(|x|^{-1/p})(1)&= \sup_{t>0}\frac{1}{\sqrt{4\pi t}}\int_{-\infty}^\infty |1-y|^{-1/p}e^{-y^2/(4t)}\dd y \\
    &\geq \sup_{t>0}\frac{1}{\sqrt{4\pi t}} e^{-1/(4t)}\int_{-1}^1 |1-y|^{-1/p} \dd y\\
    & \geq \frac{1}{\sqrt{2\pi e}} \frac{p}{p-1} 2^{(p-1)/p}~.
\end{align*}
\qed 

A better explicit lower bound is provided by the next example.

\begin{lemma}
     For all $p\in (1,\infty)$ we have
    $$\| \mathcal{G}_\ast \|_{L^p(\mathbb{R})} \geq  \left(\frac{2p}{\pi e^p}\right)^{1/(2p)} \left(\frac{p}{p-1}\right)^{1/p}~.$$   
\end{lemma}
\proof 
Set $\gamma(x)=\frac{1}{\sqrt{4\pi }}e^{-x^2/4}.$ Then
$$\sup_{t>0}\mG_t(\gamma)(x)=\sup_{t>0} \frac{1}{\sqrt{4\pi(1+t)}}e^{-x^2/(4(1+t))}~.$$
We set $u:=1+t$ and the condition $t>0$ translates into $u>1$. We define $h(u):=\frac{1}{\sqrt{4\pi u}}e^{-x^2/(4u)}$. Then we have $h'(u)=h(u)(-\frac{1}{2u}+\frac{x^2}{4u^2})$, $h'(\frac{x^2}{2})=0,$ and $h''(\frac{x^2}{2})<0$. Note that $\frac{x^2}{2}>1$ for $|x|>\sqrt{2}$. Therefore, for $|x|>\sqrt{2}$ it holds
$$\sup_{u>1} h(u)= \frac{1}{\sqrt{2\pi |x|^2}}e^{-1/2}~,$$
while for $|x|\leq \sqrt{2}$ we have
$$\sup_{u>1} h(u)= \frac{1}{\sqrt{4\pi }}e^{-x^2/4}~.$$
Using these facts, we obtain that
\begin{align*}
\|\mG_*(\gamma)\|_{L^p(\mathbb{R})}^p &= \frac{1}{(4\pi )^{p/2}}\int_{-\sqrt{2}}^{\sqrt{2}} e^{-x^2p/4}\dd x + \frac{2}{(2\pi e)^{p/2}}\int_{\sqrt{2}}^\infty x^{-p}\dd x\\
& \geq \frac{2\sqrt{2}}{(4\pi  e)^{p/2}} \left( 1 + \frac{1}{p-1}\right)~.
\end{align*}
Combining with $\|\gamma\|_{L^p(\mathbb{R})}^p = \frac{1}{(4\pi )^{p/2}}\sqrt{\frac{4\pi }{p}}$, leads to
$$\|\mG_\ast\|_{L^p(\mathbb{R})}^p \geq \sqrt{\frac{2p}{\pi e^p}}\left( \frac{p}{p-1} \right)~.$$
\qed

 For $p\in(1,\infty)$ we set \[c_\sharp(p):=\left(\frac{2p}{\pi e^p}\right)^{1/(2p)} \left(\frac{p}{p-1}\right)^{1/p} \frac{p-1}{p}\]
and let
\[
c=\inf_{p\in (1,\infty)}c_\sharp(p).
\]
It follows from the previous lemma that 
$$\|\mG_\ast\|_{L^p(\mathbb{R})}\geq \max \left\lbrace \ c \frac{p}{p-1} , 1 \right\rbrace~.$$
This, combined with Proposition \ref{propo2:GausMonotonicity}, concludes the proof of Theorem \ref{propo:GausMonotonicity} apart from the lower bound $c>0.4$ which is the content of the lemma below. 
\begin{lemma}
    \label{lem: c>0.4}
The constant $c$ is larger than $0.4.$
\end{lemma}
\proof[Proof (sketch)]
The proof is standard, however, we provide details for the convenience of the reader. Denoting $x=(p-1)/p,$ $x\in (0,1)$ our goal is to show that
\[\inf_{x\in (0,1)} \bigg(\frac{2}{\pi(1-x)}\bigg)^{(1-x)/2}x^x>0.4\sqrt{e}.\]
This will follow if we show that the infimum $h_{\rm inf}$ of the function 
\[h(x)=\frac{(1-x)}{2}\big(\log (2/\pi)-\log(1-x))+x\log x,\qquad x\in (0,1)\]
is larger than $\log(0.4)+1/2.$ Since
\[h'(x)=\frac{1}{2}\log(1-x)+\log x+\frac{3}{2}-\frac{1}{2}\log(2/\pi)\]
we see that $h'(x)=0$ if and only if  $x\sqrt{1-x}=e^{-3/2}(2/\pi)^{1/2}$ which has two solutions $0<x_1<x_2<1.$ Furthermore $h''(x)=(1-3x/2)/(x(1-x))$ is positive for $x\in(0,2/3)$ and negative for $x\in (2/3,1)$ and thus $h$ achieves its infimum in $x_1,$ i.e.\ $h_{\rm inf}=h(x_1).$ One may also verify that $x_1\in (0.198,0.2)$ and thus using the fact that $x \log x$ decreases on $(0,0.3)$ we have 
\begin{align*}
h_{\rm inf}&= h(x_1)>\frac{(1-0.2)}{2}\big(\log (2/\pi)-\log(1-0.198))+0.2\log 0.2\\
& > -0.415>\log(0.4)+1/2
\end{align*}
as desired.
\qed

\section{ Concluding remarks - radial functions}\label{sec:ConcludingRemarks}

We conclude by discussing three natural examples previously considered in the literature. Namely, we consider as input a homogeneous radial function, the characteristic function of the unit ball, and a Gaussian function, and we study the behavior, as $d\rightarrow\infty$, of the norm on $L^p(\mathbb{R}^d)$ of the maximal functions applied to these functions. These three examples seem to suggest that the limit $\lm(p)$ may be equal to one when restricting to radial inputs only which we leave as Question \eqref{eq: limrad 1}.

Our first example is the function \[\varphi_d(x)=|x|^{-d/p},\qquad x\in \R^d\setminus\{0\}.\] This function is of interest because it is an eigenfunction of $\mathcal{B}_\ast$ (as well as of $\mG_\ast$ and $\mS_\ast$). In fact, $\mathcal{B}_\ast \varphi_d (x)= b_{p,d} \varphi_d(x)$ with $b_{p,d}:=\mathcal{B}_\ast(|x|^{-d/p})(e_1)$ and $e_1:=(1,0,...,0)$. It has been shown in \cite[Eq.~3.1]{CG1} and  \cite[Theorem~3.2]{DSS} that  $\| \mathcal{B}_\ast \|_{L^p(\mathbb{R}^d)} \geq b_{p,d}$. Moreover, it has been shown in \cite[Lemma~3.4]{DSS} that $b_{p,d}>1$ whenever $p<\tfrac{d}{d-2}$ while it has been suggested in \cite[p.~478]{DSS} that $b_{p,d}$ may be identically one whenever $p\geq \tfrac{d}{d-2}$. In our analysis we will focus on $s_{p,d}:=\mS_*(|x|^{-d/p})(e_1),$ which is clearly larger than or equal to $b_{p,d}.$ Integration in polar coordinates shows that
\[
s_{p,d}=\sup_{r>0}
\frac{\omega_{d-2}}{\omega_{d-1}} \int_{-1}^{1} (1 - 2r x + r^2)^{-\frac{d}{2p}} (1 - x^2)^{\frac{d - 3}{2}} \, dx ~,
\]
where the symbol $\omega_{d-1}$ above denotes the surface area of the unit sphere in $\R^d.$
Let $p\in (1,\infty)$ be fixed and denote by $d_\sharp(p)$ the smallest $d\in\mathbb{N}$ for which $p\geq \frac{d}{d-2}$ holds.
\begin{lemma}
    Let $p\in (1,\infty)$ be fixed and take $d\geq \max \lbrace 3, d_\sharp(p) \rbrace.$ Then $s_{p,d}=1$.
\end{lemma}
\proof
In the considered range for $(p,d)$ the function $|x|^{-d/p}$ is superharmonic on $\mathbb{R}^d\setminus \lbrace 0 \rbrace$ beacuse \[\Delta(|x|^{-d/p})=\tfrac{d}{p}(2-d(p-1)/p)|x|^{-d/p-2}\leq 0.\] It follows from the properties of superharmonic functions that $\mathcal{S}_r(\varphi_d)(e_1)\leq \varphi_d(e_1)=1$ for all $r<1$. Define
$$g(r):=\int_{-1}^1(1-2rx+r^2)^{-\tfrac{d}{2p}}(1-x^2)^{\tfrac{d-3}{2}}\dd x$$
so that $\mathcal{S}_r(|x|^{-d/p})(e_1)=\tfrac{\omega_{d-2}}{\omega_{d-1}}\, g(r)$. We are interested in studying $\sup_{r>0}g(r)$. Observe that, since $\lim_{r\rightarrow\infty}g(r)=0$, the maximum is attained. We compute
$$g'(r)=-\frac{d}{2p}\int_{-1}^1(1-2rx+r^2)^{-\tfrac{d}{2p}-1}(2r-2x)(1-x^2)^{\tfrac{d-3}{2}}\dd x~.$$
We observe that $g'(1)<0$, so $r=1$  is not a maximum. Similarly, $g'(r)<0$ for all $r>1$ and therefore the maximum is not attained in this region. 
Then the maximum must be attained at some $r_p$ with $r_p<1$. It follows from the aforementioned properties of superharmonic functions that $g(r_p)\leq \frac{\omega_{d-1}}{\omega_{d-2}}$~. 
Hence, $\mathcal{S}_\ast(|x|^{-d/p})(e_1)\leq 1$ and, in view of the trivial bound $\varphi_d(x)\leq \mathcal{S}_\ast \varphi_d(x)$, the result in the statement follows.
\qed

As a consequence we have that, for $p\geq \frac{d}{d-2}$, $d\geq 3$, $b_{p,d}=1$ as suggested in \cite{DSS}.  This appears to be in line with \cite[Theorem~1]{K01} which asserts that $\mB_\ast$ has a non-constant fixed point if and only if $p\geq \frac{d}{d-2}$, $d\geq 3$. Moreover, \cite[Remark~1]{K01} observes that $f\in L^1_{loc}(\mathbb{R}^d)$ is a fixed point for $\mB_\ast$ if and only if $f$ is superharmonic.
As $\Vert\mathcal{B}_\ast \Vert_{L^p(\mathbb{R}^d)}>1$ for all fixed $p\in (1,\infty)$ and $d\geq 1$, the previous example is probably not very enlightening for our discussion. 

 Our second example concerns the characteristic function of the unit ball $B$ in $\R^d:$ 
 \[\chi_d(x)=\mathbf{1}_{B}(x),\qquad x\in \R^d.\] It has been observed in \cite[Remark 2.9]{APL} that for fixed $(p,d)$
\begin{equation*}
\frac{\|\mB_{*}\chi_d\|_{L^p(\R^d)}}{\|\chi_d \|_{L^p(\R^d)}}\ge \left(1+\frac{1}{2^{dp}(p-1)}\right)^{1/p} >1~,
\end{equation*}
hence providing the lower bound for $\|\mathcal{B}_\ast\|$ in equation \eqref{eq: mB bel APL}. The next lemma implies that
\[
\lim_{d\to \infty} \frac{\|\mB_{*}\chi_d\|_{L^p(\R^d)}}{\|\chi_d \|_{L^p(\R^d)}}=1.
\]
\begin{lemma}
\label{lem: spher ball}
   For each $p\in(1,\infty)$ we have
    $$\lim_{d\rightarrow\infty}\frac{\|\mathcal{S}_\ast \chi_d\|_{L^p(\R^d)}}{\|\chi_d\|_{L^p(\R^d)}}=1~.$$
\end{lemma}
In the proof of this lemma we will use a formula for the surface area of a spherical cap on a sphere of radius $r.$ Namely, denoting by $\theta\in[0,\pi/2]$  the polar angle between the pole and the edge of the cap one has
\begin{equation}
\label{eq: Adr form}
A_{d,r}^{\text{cap}}(\sin(\theta))=\frac{1}{2}r^{d-1}\frac{2\pi^{d/2}}{\Gamma(\tfrac{d}{2})}I_{\sin^2(\theta)}\left(\frac{d-1}{2},\frac{1}{2}\right)~,\end{equation}
where $I_x(a,b)$ is the regularized incomplete beta function,
$$I_x(a,b):=\frac{\Gamma(a+b)}{\Gamma(a)\Gamma(b)}\int_0^x t^{a-1}(1-t)^{b-1}\dd t~,$$
see for example \cite{L11}.
\proof[Proof of Lemma \ref{lem: spher ball}]
We start with proving that 
\begin{equation}
\label{eq: S on chid}
\mathcal{S}_\ast \chi_d(x)= \chi_d(x)+\frac{1}{2} I_{(|x|^2+1)^{-1}}\left(\frac{d-1}{2},\frac{1}{2}\right)
\mathbf{1}_{\mathbb{R}\setminus B}(x).
\end{equation}
The formula is clear for $x\in B$ thus we may take $x\not \in B.$ Let $|(x+r{\mathbb{S}}^{d-1})\cap B|$ be the area of the spherical cap on the sphere of radius $r$ centered at $x$ which is created by intersecting this sphere with the unit ball. Note that this area is non-zero only for $|x|-1<r<|x|+1$. For such $r$ using \eqref{eq: Adr form} we obtain
\begin{equation}
\label{eq: mSr I}
\mS_r (\chi_d)(x)=\frac{|(x+r{\mathbb{S}}^{d-1})\cap B|}{r^{d-1}\sigma(\mathbb{S}^{d-1})}=\frac{1}{2} I_{\sin^2 \theta}\left(\frac{d-1}{2},\frac{1}{2}\right).
\end{equation}
where $\theta=\theta(r,x)$ is   the polar angle between the pole and the edge of the cap $(x+r{\mathbb{S}}^{d-1})\cap B.$ For $x\not \in B$ we want to maximize $\mS_r (\chi_d)(x)$ under the constraints $|x|-1<r<|x|+1.$ In view of \eqref{eq: mSr I} this boils down to determining the maximal value of $\sin\theta.$  Let $y$ be any point on $(x+r{\mathbb{S}}^{d-1})\cap \mathbb{S}^{d-1}.$ Then, $y,x,0$ form a triangle with side lengths $1,|x|,r$ and the angle $\theta$ between the sides of length $|x|$ and $r.$ Under our constraints $\sin \theta$ is the largest when the sides of length $1$ and $|x|$ are perpendicular in which case $\sin \theta=(|x|^2+1)^{-1/2}.$ This justifies \eqref{eq: S on chid}.

Now, using \eqref{eq: S on chid} we have 
\begin{align}\label{eq:max_on_indicator_fun_bound}
    \frac{\|\mathcal{S}_\ast \chi_d\|_{L^p(\R^d)}^p}{\|\chi_d\|_{L^p(\R^d)}^p}= 1 +\frac{1}{|B|}\int_{|x|\geq 1}\left|\frac{1}{2} I_{(|x|^2+1)^{-1}}\left(\frac{d-1}{2},\frac{1}{2}\right)\right|^p \dd x~.
\end{align}
For $|x|\geq 1$ can estimate the regularized beta function in the last display as
\begin{align*}
    I_{(|x|^2+1)^{-1}}\left(\frac{d-1}{2},\frac{1}{2}\right) & =\frac{\Gamma(\tfrac{d}{2})}{\Gamma(\tfrac{d-1}{2})\sqrt{\pi}} \int_0^{(|x|^2+1)^{-1}}t^{\tfrac{d-3}{2}}(1-t)^{-1/2}\dd t\\
    & \leq \frac{\Gamma(\tfrac{d}{2})}{\Gamma(\tfrac{d-1}{2})\sqrt{\pi}} \frac{1}{\left( 1- \tfrac{1}{|x|^2+1}\right)^{1/2}} \int_0^{(|x|^2+1)^{-1}}t^{\tfrac{d-3}{2}}\dd t \\
    & \leq  \frac{\Gamma(\tfrac{d}{2})\sqrt{2}}{\Gamma(\tfrac{d-1}{2})\sqrt{\pi}} \frac{2}{d-1}(|x|^2+1)^{-\tfrac{d-1}{2}}~.
\end{align*}
Hence, integrating in polar coordinates and using \eqref{eq:max_on_indicator_fun_bound} we get
\begin{align*}
     \frac{\|\mathcal{S}_\ast \chi_d\|_{L^p(\R^d)}^p}{\|\chi_d\|_{L^p(\R^d)}^p} & \leq 1 +\frac{\Gamma(\tfrac{d}{2}+1)}{\pi^{d/2}} \frac{2\pi^{d/2}}{\Gamma(\tfrac{d}{2})} \left( \frac{\Gamma(\tfrac{d}{2})}{\Gamma(\tfrac{d-1}{2})\sqrt{\pi}} \frac{\sqrt{2}}{d-1}\right)^p \int_1^\infty (r^2+1)^{-\tfrac{p}{2}(d-1)} r^{d-1}\dd r\\
     & \leq 1 +\frac{2\Gamma(\tfrac{d}{2}+1) }{\Gamma(\tfrac{d}{2})} \left( \frac{\Gamma(\tfrac{d}{2})}{\Gamma(\tfrac{d-1}{2})\sqrt{\pi}} \frac{\sqrt{2}}{d-1}\right)^p\frac{1}{d(p-1)-p}~.
\end{align*}
Using Stirling's formula \eqref{eq: Stirling} we see that the limit as $d\rightarrow\infty$ of the quantity on the right-hand-side of the last display is equal to one. In view of the trivial bound $\mS_*(\chi_d)\ge \chi_d$ the result follows.
\qed

Our last example covers the case of the Gaussian input function
\[
\gamma_{d}(x)=\frac{1}{(4\pi )^{d/2}}e^{-|x|^2/4},\qquad x\in \R^d.
\]
\begin{lemma}
    For each  $p\in(1,\infty)$ we have
     $$\lim_{d\rightarrow\infty}\frac{\|\mG_\ast \gamma_{d}\|_{L^p(\R^d)}}{\|\gamma_{d}\|_{L^p(\R^d)}}=\lim_{d\rightarrow\infty}\frac{\|\mB_\ast \gamma_{d}\|_{L^p(\R^d)}}{\|\gamma_{d}\|_{L^p(\R^d)}}=\lim_{d\rightarrow\infty}\frac{\|\mS_\ast \gamma_{d}\|_{L^p(\R^d)}}{\|\gamma_{d}\|_{L^p(\R^d)}}=1~.$$
\end{lemma}
\proof By \eqref{eq: diff Lp} from Theorem \ref{thm: limit} it suffices to focus on the Gaussian maximal function and prove that
\[
\lim_{d\rightarrow\infty}\frac{\|\mG_\ast \gamma_{d}\|_{L^p(\R^d)}}{\|\gamma_{d}\|_{L^p(\R^d)}}=1.
\]

Since $\gamma_d$ is a Gaussian we have 
$$\sup_{t>0} \mG_t (\gamma_d)(x)=\sup_{t>0}\frac{1}{(4\pi(1+t))^{d/2}}e^{-|x|^2/(4(1+t))}~.$$
We set $u:=1+t$ which translates the condition $t>0$ into $u>1$. Note that the function $h(u):= \frac{1}{(4\pi u)^{d/2}}e^{-|x|^2/(4u)}$
satisfies $h'(u)=h(u)(-\frac{d}{2u}+\frac{|x|^2}{4u}),$ $h'(\frac{|x|^2}{2d})=0$ and $h''(\frac{|x|^2}{2d})\leq 0$. Since $\frac{|x|^2}{2d}>1$ for $|x|>\sqrt{2d}$ we see that
$$\sup_{u>1} h(u)=\left( \frac{d}{2\pi|x|^2}\right)^{d/2}e^{-d/2} \qquad \text{for}\; |x|>\sqrt{2d}~,$$
while for $|x|\leq \sqrt{2d}$ the supremum is attained at the boundary $u=1$, namely
$$\sup_{u>1} h(u)=\frac{1}{(4\pi )^{d/2}}e^{-|x|^2/4}=\gamma_d(x) \qquad \text{for}\; |x|\leq \sqrt{2d}~.$$
Using these facts, we obtain that
\begin{align*}
    \|\sup_{t>0} \mG_t (\gamma_d)\|_{L^p(\mathbb{R}^d)}^p &= \int_{|x|\leq \sqrt{2d}} |\gamma_d(x)|^p\dd x + \left(\frac{d}{2\pi e} \right)^{dp/2}\int_{|x|> \sqrt{2d}} |x|^{-dp}\dd x \\
    & =\int_{|x|\leq \sqrt{2d}} |\gamma_d(x)|^p\dd x  + \left(\frac{d}{2\pi e} \right)^{dp/2} \frac{2\pi^{d/2}}{\Gamma(\frac{d}{2})} \frac{1}{(2d)^{\frac{d}{2}(p-1)}} \frac{1}{d(p-1)}~.
\end{align*}
We compute $\|\gamma_d\|_{L^p(\mathbb{R})^d}^p=\frac{1}{(4\pi )^{dp/2}}(\frac{4\pi }{p})^{d/2}$. Hence, to conclude, we are left to show that, for any fixed $p\in(1,\infty)$, the following limit is equal to zero,
\begin{align*}
    \lim_{d\rightarrow\infty} \;\; (4\pi )^{dp/2} \left(\frac{p}{4\pi }\right)^{d/2}&\left( \left(\frac{d}{2\pi e} \right)^{dp/2} \frac{2\pi^{d/2}}{\Gamma(\frac{d}{2})} \frac{1}{(2d)^{\frac{d}{2}(p-1)}} \frac{1}{d(p-1)} \right)\\
    & = \lim_{d\rightarrow\infty} \left(\frac{1}{e}\right)^{dp/2}\left( \frac{pd}{2}\right)^{d/2}\frac{2}{\Gamma(\frac{d}{2})}\frac{1}{d(p-1)} =: \lim_{d\rightarrow\infty} v(p,d)~.
\end{align*}
We observe that for any fixed $d\geq 1$, $v(p,d)$ is a decreasing function of  $p\in (1,\infty)$. Therefore, for any fixed $p$ and sufficiently small, fixed, $\varepsilon>0$, using Stirling's approximation \eqref{eq: Stirling} we have 
\begin{align*}
   \lim_{d\rightarrow\infty} v(p,d)\leq  \lim_{d\rightarrow\infty}
   \left(\frac{1}{e}\right)^{(1+\varepsilon)d/2}\left( \frac{(1+\varepsilon)d}{2}\right)^{d/2}\frac{2}{\Gamma(\frac{d}{2})}\frac{1}{d\varepsilon} = 0~,
\end{align*}
as claimed.
\qed

\end{document}